\documentclass[reqno, 11pt, a4paper]{amsart}
\usepackage{amsfonts, amsmath, amssymb, amsthm, color}
\usepackage[hmargin=2.75cm, vmargin=2.9cm]{geometry}

\newtheorem{thm}{Theorem}[section]
\newtheorem{lemma}[thm]{Lemma}
\newtheorem{prop}[thm]{Proposition}
\newtheorem{cor}[thm]{Corollary}

\theoremstyle{definition}
\newtheorem{rmk}[thm]{Remark}
\newtheorem{defn}[thm]{Definition}

\newcommand{\ep}{\epsilon}
\newcommand{\vep}{\varepsilon}

\newcommand{\pa}{\partial}

\newcommand{\B}{\mathbb{B}}
\newcommand{\N}{\mathbb{N}}
\newcommand{\R}{\mathbb{R}}
\renewcommand{\S}{\mathbb{S}}

\newcommand{\mcc}{\mathcal{C}}

\newcommand{\mci}{\mathcal{I}}
\newcommand{\mcm}{\mathcal{M}}
\newcommand{\mcn}{\mathcal{N}}
\newcommand{\mcp}{\mathcal{P}}

\newcommand{\tif}{\tilde{f}}
\newcommand{\tig}{\tilde{g}}
\newcommand{\tih}{\tilde{h}}
\newcommand{\wtf}{\widetilde{F}}
\newcommand{\wtu}{\widetilde{U}}
\newcommand{\wtv}{\widetilde{V}}
\newcommand{\wPh}{\widetilde{\Phi}}

\newcommand{\bx}{\bar{x}}

\newcommand{\hf}{\hat{f}}
\newcommand{\hg}{\hat{g}}

\newcommand{\ou}{\overline{U}}
\newcommand{\tni}{\textnormal{II}}

\renewcommand{\(}{\left(}
\renewcommand{\)}{\right)}

\allowdisplaybreaks
\numberwithin{equation}{section}

\begin{document}
\title[Compactness of scalar-flat conformal metrics]{Compactness of scalar-flat conformal metrics \\
on low-dimensional manifolds \\
with constant mean curvature on boundary}

\author{Seunghyeok Kim}
\address[Seunghyeok Kim]{Department of Mathematics and Research Institute for Natural Sciences, College of Natural Sciences, Hanyang University, 222 Wangsimni-ro Seongdong-gu, Seoul 04763, Republic of Korea}
\email{shkim0401@hanyang.ac.kr shkim0401@gmail.com}

\author{Monica Musso}
\address[Monica Musso]{Department of Mathematical Sciences, University of Bath, Bath BA2 7AY, United Kingdom, and Facultad de Matem\'{a}ticas, Pontificia Universidad Cat\'{o}lica de Chile, Avenida Vicu\~{n}a Mackenna 4860, Santiago, Chile}
\email{m.musso@bath.ac.uk}

\author{Juncheng Wei}
\address[Juncheng Wei] {Department of Mathematics, University of British Columbia, Vancouver, B.C., Canada, V6T 1Z2}
\email{jcwei@math.ubc.ca}

\begin{abstract}
We concern $C^2$-compactness of the solution set of the boundary Yamabe problem on smooth compact Riemannian manifolds with boundary provided that their dimensions are $4$, $5$ or $6$.
By conducting a quantitative analysis of a linear equation associated with the problem,
we prove that the trace-free second fundamental form must vanish at possible blow-up points of a sequence of blowing-up solutions.
Applying this result and the positive mass theorem, we deduce the $C^2$-compactness for all 4-manifolds (which may be non-umbilic).
For the 5-dimensional case, we also establish that a sum of the second-order derivatives of the trace-free second fundamental form is non-negative at possible blow-up points.
We essentially use this fact to obtain the $C^2$-compactness for all 5-manifolds.
Finally, we show that the $C^2$-compactness on 6-manifolds is true if the trace-free second fundamental form on the boundary never vanishes.
\end{abstract}

\date{\today}
\maketitle

\section{Introduction}
Let $(M, g)$ be an $N$-dimensional ($N \ge 3$) smooth compact Riemannian manifold with boundary $\pa M$.
Let also $\Delta_g$ be the Laplace-Beltrami operator on $M$, $R[g]$ the scalar curvature on $M$, $\nu$ the inward normal vector to $\pa M$,
and $H[g]$ be the mean curvature of $\pa M$.
In \cite{Es}, Escobar asked if $(M,g)$ can be conformally deformed to a scalar-flat manifold with boundary of constant mean curvature.
This problem, which we will call the boundary Yamabe problem, can be understood as a generalization of the Riemann mapping theorem
and is equivalent to finding a positive smooth solution to a nonlinear boundary value problem with critical exponent
\begin{equation}\label{eq_Yamabe}
\begin{cases}
L_g U = 0 &\text{in } M,\\
B_g U = Q(M,\pa M) U^{N \over N-2} &\text{on } \pa M.
\end{cases}
\end{equation}
Here $L_g$ is the conformal Laplacian and $B_g$ is the associated conformal boundary operator defined as
\[L_g = -\Delta_g + \frac{N-2}{4(N-1)} R[g] \quad \text{and} \quad
B_g = -\frac{\pa}{\pa \nu} + \frac{N-2}{2} H[g],\]
and $Q(M,\pa M)$ is a constant whose sign is determined by the conformal structure of $M$.

Weak solutions to \eqref{eq_Yamabe} correspond to critical points of the functional
\[Q(U) = \frac{\int_M (|\nabla_g U|_g^2 + \frac{N-2}{4(N-1)} R[g] U^2) dv_g + \int_{\pa M} H[g] U^2 dv_h}{(\int_{\pa M} |U|^{2(N-1) \over N-2} dv_h)^{N-2 \over N-1}}\]
defined for an element $U$ in the Sobolev space $H^1(M)$ with $U \ne 0$ on $\pa M$,
where $\nabla_g$ represents the gradient on $(M,g)$, $h$ is the restriction of the metric $g$ on $\pa M$,
and $dv_g$ and $dv_h$ are the volume form on $M$ and on $\pa M$, respectively.
Escobar \cite{Es} proved that the Sobolev quotient
\[Q(M,\pa M) = \inf \left\{Q(U): U \in H^1(M),\, U \ne 0 \text{ on } \pa M \right\}\]
attains its minimizer if $Q(M,\pa M) < Q(\B^N,\pa \B^N)$ where the unit ball $\B^N = \{x \in \R^N: |x| < 1\}$ is endowed with the Euclidean metric.
This is analogous to the observation of Aubin \cite{Au} for the classical Yamabe problem.

Thanks to the effort of several researchers, the existence of a solution to \eqref{eq_Yamabe} is now well-established:
Escobar \cite{Es, Es2}, Marques \cite{Ma2, Ma3}, Almaraz \cite{Al2} and Chen \cite{Ch} found a minimizer of the functional $Q$ for almost all manifolds.
By applying the barycenter technique of Bahri and Coron, Mayer and Ndiaye \cite{MN} covered all the remaining cases.
Regularity property of \eqref{eq_Yamabe} was investigated by Cherrier \cite{Cr}.

Concerning multiplicity of solutions to \eqref{eq_Yamabe}, the only interesting case is when $Q(M,\pa M) > 0$.
If $Q(M,\pa M) < 0$, the conformal covariance of the operators $L_g$ and $B_g$ shows that \eqref{eq_Yamabe} has only one solution. 
If $Q(M,\pa M) = 0$, it is a linear equation and its solution is unique up to positive multiplicative constants.
On the other hand, the case that $M$ is conformally equivalent to the unit ball $\B^N$ (so that $Q(M,\pa M) = Q(\B^N,\pa \B^N) > 0$) is special,
and the solution set of \eqref{eq_Yamabe} was completely classified thanks to the works of Escobar \cite{Es3} and Li and Zhu \cite{LZhu2}; see Subsection \ref{subsec_bubble}.

In about two decades, several results on $C^2(M)$-compactness of the solution set of \eqref{eq_Yamabe} appeared under the assumption that $Q(M,\pa M) > 0$.
Felli and Ould Ahmedou \cite{FO, FO2} deduced compactness results for locally conformally flat manifolds and 3-manifolds provided that their boundaries are umbilic.
Very recently, the umbilicity condition was lifted for 3-manifolds by Almaraz et al. \cite{AdQW}.
If the dimension $N$ of the manifold $M$ satisfies $N \ge 7$ and the trace-free second-fundamental form on $\pa M$ is nonzero everywhere,
the result of Almaraz \cite{Al} shows that the $C^2(M)$-compactness continues to hold.
If either $N > 8$ and the Weyl tensor of $M$ never vanishes on $\pa M$, or $N = 8$ and the Weyl tensor of $\pa M$ never vanishes on $\pa M$,
the $C^2(M)$-compactness is still true for manifolds $M$ with umbilic boundary, as shown by Ghimenti and Micheletti \cite{GM}.

Compactness results for other boundary Yamabe-type problems can be found in Han and Lin \cite{HL}, Djadli et al. \cite{DMO, DMO2}, Disconzi and Khuri \cite{DK}, and so on.
By using the compactness property, C\'adenas and Sierra \cite{CS} yielded uniqueness of solutions to \eqref{eq_Yamabe} for some manifolds whose metrics are non-degenerate.

As far as the authors know, compactness results on \eqref{eq_Yamabe} have been known only for manifolds with boundary of dimension $N = 3$ or $N \ge 7$, unless manifolds are locally conformally flat.
The main purpose of this paper is to treat all manifolds with boundary of dimension $N = 4$ and $5$, and generic manifolds with boundary of dimension $N = 6$.
\begin{thm}\label{thm_main}
For $N = 4, 5, 6$, let $(M, g)$ be an $N$-dimensional smooth compact Riemannian manifold with boundary $\pa M$
such that $Q(M, \pa M) > 0$ and $M$ is not conformally equivalent to the unit ball $\B^N$.
If $N = 6$, we also assume that the trace-free second-fundamental form never vanishes on $\pa M$.
Then, for any $\vep_0 > 0$ small, there exists a constant $C > 1$ depending only on $M, g$ and $\vep_0$ such that
\[C^{-1} \le U \le C \quad \text{on } M \quad \text{and} \quad \|U\|_{C^2(M)} \le C\]
for any solution $U \in H^1(M)$ to
\begin{equation}\label{eq_Yamabe_2}
\begin{cases}
L_g U = 0 &\text{in } M,\\
B_g U = Q(M,\pa M) U^p &\text{on } \pa M
\end{cases}
\end{equation}
with $p \in [1+\vep_0, \frac{N}{N-2}]$.
\end{thm}
\noindent The transversality argument shows that if $N \ge 4$, the set of metrics on $M$ whose trace-free second fundamental form on $\pa M$ vanishes nowhere is open and dense in the space of all Riemannian manifolds on $M$.
This justifies the terminology `generic' used above.
Also, as can be observed in Theorem \ref{thm_main}, we will deal with a slightly generalized equation \eqref{eq_Yamabe_2} instead of \eqref{eq_Yamabe}.

\medskip
Our strategy follows the argument in the lecture note \cite{Sc} of Schoen
where he raised the question of $C^2$-compactness of the solution set of the classical Yamabe problem
and resolved it for locally conformally flat manifolds.
It has been further developed by Li and Zhu \cite{LZhu}, Druet \cite{Dr}, Marques \cite{Ma}, Li and Zhang \cite{LZ, LZ2} and Khuri et al. \cite{KMS}.
Furthermore, Li \cite{Li} and Li and Xiong \cite{LX} studied compactness results of the Q-curvature problem, which is the fourth-order analogue of the Yamabe problem.

Once Theorem \ref{thm_main} is established, one can deduce the existence of a solution to \eqref{eq_Yamabe} by applying the standard Leray-Schauder degree argument as in \cite{FO, HL}.
There also should exist the strong Morse inequality in our framework as in \cite[Theorem 1.4]{KMS}.

\medskip
We leave two more remarks for the theorem.
\begin{rmk}
The key idea of our main theorem is to perform a fine analysis of associated linearized equations with \eqref{eq_Yamabe_2}
in proving that the trace-free second fundamental form mush vanish at possible blow-up points of a sequence of blowing-up solutions.
Interestingly, this process is somehow related to the way that Marques \cite{Ma3} constructed test functions in his existence theorem for \eqref{eq_Yamabe} on low-dimensional manifolds with non-umbilic boundary.
Indeed, his test functions consist of not only truncated bubbles but also some additive correction terms.
This is a distinctive feature of the boundary Yamabe problem compared with the classical one.

Our argument can be be further applied in the following settings.
\begin{enumerate}
\item Based on the existence results of Marques \cite{Ma2} and Almaraz \cite{Al2} for \eqref{eq_Yamabe} on manifolds with umbilic boundary,
    we expect that one can lower the threshold dimension 8 in the aforementioned compactness theorem of Ghimenti and Micheletti \cite{GM} to 6.
\item As a matter of fact, the boundary Yamabe problem can be seen as the special case of the fractional Yamabe problem
    where the symbol of the differential operator is the same as that of the half-Laplacian.
    In \cite{KMW}, we proved that the solution set of the fractional Yamabe problem is $C^2$-compact on conformal infinities of asymptotically hyperbolic manifolds,
    under the assumptions that the dimension is sufficiently high and the second-fundamental form never vanishes.
    In view of our existence result \cite{KMW2}, we expect that the compactness result holds for conformal infinities of dimension $\ge 4$ as far as the same geometric condition is maintained.
\item To examine stability issue under small perturbation of \eqref{eq_Yamabe},
    Ghimenti et al. \cite{GMP, GMP2} constructed blowing-up solutions when the linear perturbation of the mean curvature on the boundary is strictly positive everywhere;
    see also Deng et al. \cite{DKP} where analogous results were derived in the setting of the fractional Yamabe problem.
    In building suitable approximation solutions, they had to analyze an associated linearized equation with \eqref{eq_Yamabe} which is essentially the same as ours.
    Due to this reason, their results require some dimensional assumptions.
    Our method can allow one to treat lower-dimensional cases.
\end{enumerate}
\end{rmk}

\begin{rmk}\label{rmk_mar}
The proof of the main theorem shows that remarkable phenomena happen on 5-manifolds $(M, g)$ with boundary.
\begin{enumerate}
\item In Lemma \ref{lemma_pmt} and the first paragraph of Section \ref{sec_cpt},
    we construct an asymptotically flat manifold $(M \setminus \{y_0\}, G^{4 \over N-2}g)$, which we call the conformal blow-up of $(M, g)$.
    Corollary \ref{cor_mass} reveals that its mass is involved with not only the Green's function (defined in \eqref{eq_Green}) but also the trace-free second fundamental form on the boundary $\pa M$.
    Therefore, the mass carries global and local information simultaneously.
    This is in striking contrast with manifolds without boundary
    in that the mass of their conformal blow-ups depend only on the Green's function, namely, global information.
\item In Subsection \ref{subsec_nonneg}, we will see that the sign of the local information of mass is encoded in the $\ep^3 |\log \ep|$-order of the expansion
    of a local Pohozaev identity \eqref{eq_poho_2} with respect to a small scaling parameter $\ep > 0$.
    This is totally different from the classical Yamabe problem.
    In the classical one, the order involving the logarithm carries meaningful geometric information only if the manifold is even-dimensional.
\end{enumerate}
\end{rmk}

In \cite{Al3}, Almaraz constructed manifolds with umbilic boundary of dimension $N \ge 25$ on which the solution set of \eqref{eq_Yamabe} is $L^{\infty}$-unbounded (in particular, $C^2$-noncompact).
In view of the full compactness result of Khuri et al. \cite{KMS} and the non-compactness results of Brendle \cite{Br} and Brendle and Marques \cite{BM} for the classical Yamabe problem,
a natural expectation is that the solution set of \eqref{eq_Yamabe} is $C^2$-compact for all manifolds with boundary of dimension $N \le 24$ under the validity of the positive mass theorem.
However, although Schoen's argument in \cite{Sc} works in principle and we develop several efficient methods for the boundary Yamabe problem in this paper, fully achieving this seems still a difficult task.

To establish the $C^2$-compactness result for general manifolds of high dimension, we must prove that
the trace-less second fundamental form and the Weyl tensor vanish up to some high order at each blow-up point.
This requires a very accurate pointwise estimate of blowing-up solutions, which can be achieved only if one has a good understanding of linearized equations.
In the analysis on the classical Yamabe problem, Khuri et al. \cite{KMS} observed that
solutions of their linearized problems can be written explicitly in the form of rational functions.
Unfortunately, the boundary Yamabe problem seems not to have a similar property.

On the other hand, we may also need a quite precise control of the Green's function $G$ of the conformal Laplacian with Neumann boundary condition; see \eqref{eq_Green} of its definition.
In our analysis, we only need a rough control of $G$ (described in Lemma \ref{lemma_pmt}) as in the proof of the compactness theorem for 3-dimensional manifolds \cite{AdQW}.

\medskip
The rest of the paper is organized as follows:
\begin{itemize}
\item[-] In Section \ref{sec_pre}, we recall some analytic and geometric tools which we need throughout the proof of Theorem \ref{thm_main}.
    These include the expansion of the metric in Fermi coordinates, definition of the bubbles,
    a local Pohozaev's identity and the positive mass theorem on asymptotically flat manifolds with boundary.
\item[-] In Section \ref{sec_blow}, we characterize blow-up points of solutions to \eqref{eq_Yamabe_2}
    and provide basic qualitative properties of solutions near blow-up points.
\item[-] In Section \ref{sec_lin}, we study a linearized equation associated with \eqref{eq_Yamabe_2} arising from the first-order expansion of the metric.
    In order to treat low-dimensional manifolds, we need to understand its solution more precisely than higher-dimensional cases.
    For this aim, we decompose the solution into two pieces and analyze them quantitatively.
    This is one of the key parts of the proof. We also perform a refined blow-up analysis.
\item[-] In Section \ref{sec_van}, we carry out the proof of the vanishing theorem of the trace-free second fundamental form at any isolated simple blow-up point.
    For 5-manifolds, we also establish that a sum of the second-order derivatives of the trace-free second fundamental form is non-negative at each isolated simple blow-up point.
    These results are based on the quantitative analysis of the linearized equation conducted in the previous section.
\item[-] In Section \ref{sec_lsr}, employing the vanishing theorem, we prove a local Pohozaev sign condition that guarantees that every blow-up point is isolated simple.
\item[-] In Section \ref{sec_cpt}, by applying the positive mass theorem, we conclude that
    the solution set of \eqref{eq_Yamabe_2} is $C^2$-compact for every 4- and 5-manifold unless it is conformally equivalent to the unit ball.
    For 6-manifolds, we also show that the $C^2$-compactness of the solution set holds provided that the trace-free second fundamental form on the boundary never vanishes.
\item[-] In Appendix \ref{sec_app_PU}, we provide technical arguments regarding the two pieces of the solutions to the linearized equation to \eqref{eq_Yamabe_2}.
\end{itemize}
To elucidate our method, we will omit most of the proofs of intermediate results which closely follow the corresponding ones in similar settings, leaving appropriate references instead.

\bigskip \noindent \textbf{Notations.}

\medskip \noindent - Let $n = N-1$. Moreover, for any $x \in \R^N_+ = \{(x_1, \cdots, x_n, x_N) \in \R^N: x_N > 0\}$, we denote $\bx = (x_1, \cdots, x_n) \in \R^n$.
We often identify $\bx \in \R^n$ and $(\bx,0) \in \pa \R^N_+$.

\medskip \noindent - We will sometimes use $\pa_a = \frac{\pa}{\pa x_a}$, $\pa_{ab} = \frac{\pa^2}{\pa x_a \pa x_b}$, etc.

\medskip \noindent - Given $x \in \R^N_+$, $\bx \in \R^n$ and $r > 0$, let $B^N_+(x,r)$ be the $N$-dimensional upper half-ball centered at $x$ of radius $r$,
and $B^n(\bx,r)$ the $n$-dimensional ball centered at $\bx$ of radius $r$.
We often identify $B^n(\bx,r)$ and $\pa B^N_+((\bx,0),r) \cap \pa \R^N_+$.
Set $\pa_I B^N_+((\bx,0),r) = \pa B^N_+((\bx,0),r) \cap \R^N_+$.

\medskip \noindent - $S$ represents a surface measure. Its subscript $x$ or $\bx$ denotes the dependent variables.

\medskip \noindent -  $D^{1,2}(\R^N_+)$ is the homogeneous Sobolev space in $\R^N_+$ defined as
\[D^{1,2}(\R^N_+) = \left\{U \in L^{2N \over N-2}(\R^N_+): \nabla U \in L^2(\R^N_+)\right\}.\]

\medskip \noindent - $|\S^{n-1}|$ is the surface area of the unit $(n-1)$-sphere $\S^{n-1}$.

\medskip \noindent - The metric $h$ on the boundary $\pa M$ of the Riemannian manifold $(M,g)$ is the restriction of the metric $g$ to $\pa M$.

\medskip \noindent - For any $y \in \pa M$ and $r > 0$ small, $B_g(y,r)$ and $B_h(y,r)$ stand for the geodesic half-ball on $(M,g)$ and the geodesic ball on $(\pa M, h)$, respectively.
Also, $d_g$ is the distance function on $(M,g)$.

\medskip \noindent - The Einstein summation convention for repeated indices is adopted throughout the paper.
Unless otherwise stated, the indices $i$, $j$, $k$, $l$, $m$ and $s$ always range over values from 1 to $n$, while $a$, $b$, $c$ and $d$ take values from 1 to $N$.
Also, $\delta_{ab}$ is the Kronecker delta.

\medskip \noindent - We denote by $R_{abcd}[g]$ the full Riemannian curvature tensor on $(M,g)$,
by $R_{ab}[g]$ the Ricci curvature tensor on $M$, and by $R[g]$ the scalar curvature on $M$.
The quantities $R_{ijkl}[h]$, $R_{ij}[h]$ and $R[h]$ are the corresponding curvatures defined on the boundary $(\pa M, h)$.

\medskip \noindent - We write by $\tni[g]$ the second fundamental form of $\pa M$,
by $H[g] = \frac{1}{n} h^{ij}\tni_{ij}[g]$ the mean curvature on $\pa M$,
and by $\pi[g] = \tni[g] - Hg$ the trace-free second fundamental form of $\pa M$.
Furthermore, $\|\pi[g]\|^2 = h^{ik}h^{jl}\pi_{ij}[g]\pi_{kl}[g]$ stands for the square of its norm.

\medskip \noindent - For an $r$-tensor $T$, we write
\[\text{Sym}_{i_1 \cdots i_r} T_{i_1 \cdots i_r} = \frac{1}{r!} \sum_{\sigma \in S_r} T_{i_{\sigma(1)} \cdots i_{\sigma(r)}}\]
where $S_r$ is the symmetric group over a set of $r$ symbols.

\medskip \noindent - For a multi-index $\alpha = (\alpha_1, \cdots, \alpha_n) \in \R^n$,
\begin{equation}\label{eq_mi}
|\alpha| = \sum_{i=1}^n \alpha_i,\ \alpha! = \prod_{i=1}^n \alpha_i!
\quad \text{and} \quad
\frac{\pa}{\pa x_{\alpha}} = \frac{\pa^{\alpha_1}}{\pa x_1^{\alpha_1}} \cdots \frac{\pa^{\alpha_n}}{\pa x_n^{\alpha_n}}.
\end{equation}
$\beta$, $\beta'$ and $\beta''$ also denote multi-indices.

\medskip \noindent - The letter $C$ denotes a generic positive constant that may vary from line to line.

\section{Preliminaries}\label{sec_pre}
\subsection{Metric expansion and conformal coordinates}
Fix a point $y_* \in \pa M$.
For any $y \in \pa M$ near $y_*$, let $\bx = (x_1, \cdots, x_n) \in \R^n$ be normal coordinates on $\pa M$ (centered at $y_*$) of $y$.
Denote by $\nu(y)$ the inward normal vector to $\pa M$ at $y$.
We say that $x = (\bx, x_N) \in \R_+^N$ is Fermi coordinates on $M$ (centered at $y_*$) of the point $\exp_{y}(x_N \nu(x)) \in M$.

In Lemma 2.2 of Marques \cite{Ma2}, the following expansion of the metric $g$ near $y_*$ was given.
\begin{lemma}\label{lemma_metric}
In Fermi coordinates centered at $y_* \in M$, it holds that
\[g_{ij}(x) = \delta_{ij} + A_{ij}(x) + O(|x|^4),\]
$g_{iN}(x) = 0$ and $g_{NN}(x) = 1$, where
\begin{align*}
A_{ij}(x) &= - 2\tni_{ij}[g]x_N - \frac{1}{3} R_{ikjl}[h] x_kx_l - 2\tni_{ij,k}[g] x_kx_N + (-R_{iNjN}[g] + \tni_{is}[g]\tni_{sj}[g]) x_N^2 \\
&\ - \frac{1}{6} R_{ikjl,m}[h] x_kx_lx_m
+ \(- \tni_{ij,kl}[g] + \frac{2}{3} \textnormal{Sym}_{ij}(R_{iksl}[h] \tni_{sj}[g]) \) x_kx_lx_N \\
&\ + \(- R_{iNjN,k}[g] + 2 \textnormal{Sym}_{ij}(\tni_{is,k}[g]\tni_{sj}[g]) \) x_kx_N^2 \\
&\ + \frac{1}{6} \(- 2R_{iNjN,N}[g] + 8 \textnormal{Sym}_{ij}(\tni_{is}[g]R_{jNsN}[g]) \) x_N^3.
\end{align*}
Every tensor in the expansion is evaluated at $y_*$ and commas denote covariant differentiation.
\end{lemma}

The next lemma describes the existence of conformal coordinates. Refer to Propositions 3.1 and 3.2 of \cite{Ma2}.
\begin{lemma}\label{lemma_conf}
For given a point $y_* \in M$ and an integer $\kappa \ge 2$, there exists a metric $\tig$ on $M$ conformal to $g$ such that
\begin{equation}\label{eq_kappa}
\det \tig(x) = 1 + O(|x|^{\kappa})
\end{equation}
in $\tig$-Fermi coordinates centered at $y_*$. In particular,
\begin{equation}\label{eq_conf}
H[g] = H_{,i}[g] = R_{ij}[h] = 0 \quad \text{and} \quad R_{NN}[g] = - \|\pi[g]\|^2 \quad \text{at } y_*.
\end{equation}
Moreover, $\tig$ can be written as $\tig = \omega^{4 \over N-2}g$ for some positive smooth function $w$ on $\pa M$ such that $w(y_*) = 1$ and $\nabla w(y_*) = 0$.
\end{lemma}

\subsection{Bubbles in the Euclidean half-space}\label{subsec_bubble}
Assume that $N \ge 3$.
For $\lambda > 0$ and $\xi \in \R^n$, let a bubble $W_{\lambda,\xi}$ be a function defined as
\begin{equation}\label{eq_W_lx}
W_{\lambda,\xi}(x) = \frac{\lambda^{N-2 \over 2}}{(|\bx-\xi|^2 + (x_N+\lambda)^2)^{N-2 \over 2}} \quad \text{for } x \in \R^N_+,
\end{equation}
which is an extremal function of the Sobolev trace inequality $D^{1,2}(\R^N_+) \hookrightarrow L^{2(N-1) \over N-2}(\R^n)$; see Escobar \cite{Es5}.
According to Li and Zhu \cite{LZ}, any solution to the boundary Yamabe problem on $\R^N_+$
\begin{equation}\label{eq_W_eq}
\begin{cases}
- \Delta U = 0 &\text{in } \R^N_+,\\
U > 0 &\text{in } \R^N_+,\\
- \dfrac{\pa U}{\pa x_N} = (N-2)U^{N \over N-2} &\text{on } \R^n
\end{cases}
\end{equation}
must be a bubble.
Note that a sequence $\{W_{{1 \over n},0}\}_{n \in \N}$ of bubbles exhibits a blow-up phenomenon as $n \to \infty$, and in particular, the family of all bubbles is not $L^{\infty}(\R^N_+)$-bounded.
Furthermore, D\'avila et al. \cite{DDS} proved that the solution space of the linear problem
\[\begin{cases}
- \Delta \Phi = 0 &\text{in } \R^N_+,\\
- \dfrac{\pa \Phi}{\pa x_N} = N w_{\lambda,\xi}^{2 \over N-2} \Phi &\text{on } \R^n,\\
\|\Phi(\cdot, 0)\|_{L^{\infty}(\R^n)} < \infty,
\end{cases}\]
where $w_{\lambda,\xi}(\bx) = W_{\lambda,\xi}(\bx,0)$ on $\R^n$, is spanned by
\[Z_{\lambda,\xi}^1 = {\pa W_{\lambda,\xi} \over \pa \xi_1},\ \cdots,\ Z_{\lambda,\xi}^n = {\pa W_{\lambda,\xi} \over \pa \xi_n}
\quad \text{and} \quad
Z_{\lambda,\xi}^0 = - {\pa W_{\lambda,\xi} \over \pa \lambda};\]
refer also to Lemma 2.1 of \cite{Al}.

\subsection{Conformally invariant equations}
Let $\delta = \frac{N}{N-2} - p \ge 0$. It turns out that it is more convenient to deal with the following form of the equation
\begin{equation}\label{eq_Yamabe_3}
\begin{cases}
L_g U = 0 &\text{on } M,\\
B_g U = (N-2) f^{-\delta} U^p &\text{on } \pa M
\end{cases}
\end{equation}
than \eqref{eq_Yamabe_2}. Indeed, by the conformal covariance property of the operators $L_g$ and $B_g$,
the metric $\tig = \omega^{4 \over N-2} g$ conformal to $g$ and the function $\wtu = \omega^{-1} U > 0$ on $M$ satisfy
\begin{equation}\label{eq_Yamabe_30}
\begin{cases}
L_{\tig} \wtu = 0 &\text{on } M,\\
B_{\tig} \wtu = (N-2) \tif^{-\delta} \wtu^p &\text{on } \pa M
\end{cases}
\end{equation}
where $\tif = \omega f$. Obviously, it is an equation of the same type as \eqref{eq_Yamabe_3}.

We will study a sequence $\{U_m\}_{m \in \N}$ of solutions to \eqref{eq_Yamabe_3}
with suitable choices of the exponents $p = p_m \in [1+\vep_0, \frac{N}{N-2}]$ and $\delta = \delta_m = \frac{N}{N-2} - p_m$,
the metric $g = g_m$ on $M$ and the smooth positive function $f = f_m$ on $\pa M$.
Although we postpone their specific description to Section \ref{sec_blow}, we stress that our choices will induce that
$p_m \to p_0$, $g_m \to g_0$ in $C^4(M, \R^{N \times N})$ and $f_m \to f_0 > 0$ in $C^2(\pa M)$ as $m \to \infty$, and $g_0$ is a metric on $M$.

\subsection{Pohozaev's identity}
In the analysis of blowing-up solutions, we shall rely on the following version of local Pohozaev's identity. For its derivation, see Proposition 3.1 of \cite{Al}.
\begin{lemma}
Assume that $N \ge 3$. Let $U \in H^1(B^N_+(0,\rho_1))$ be a solution to
\[\begin{cases}
- \Delta U = Q &\text{in } B^N_+(0,\rho_1),\\
- \dfrac{\pa U}{\pa x_N} + \dfrac{N-2}{2} HU = f U^p &\text{on } B^n(0,\rho_1)
\end{cases}\]
where $p \in [1, \frac{N}{N-2}]$, $Q \in L^{\infty}(B^N_+(0,\rho_1))$ and $H,\, f \in C^1(B^n(0,\rho_1))$. For any $\rho \in (0, \rho_1)$, we define
\begin{equation}\label{eq_poho_0}
\mcp'(U, \rho) = \int_{\pa_I B^N_+(0,\rho)} \left[-\({N-2 \over 2}\) U {\pa U \over \pa \nu} - {\rho \over 2} |\nabla U|^2 + \rho \left| {\pa U \over \pa \nu} \right|^2 \right] dS_x
\end{equation}
and
\begin{equation}\label{eq_poho_1}
\mcp(U, \rho) = \mcp'(U, \rho) + {\rho \over p+1} \int_{\pa B^n(0,\rho)} f U^{p+1} dS_{\bx}
\end{equation}
where $\nu$ is the inward unit normal vector with respect to $\pa_I B^N_+(0,\rho)$. Then we have
\begin{equation}\label{eq_poho_2}
\begin{aligned}
\mcp(U, \rho) &= - \int_{B^N_+(0,\rho)} Q \left[x_a \pa_a U + \({N-2 \over 2}\) U \right] dx \\
&\ + \frac{N-2}{2} \int_{B^n(0,\rho)} H \left[x_i \pa_i U + \({N-2 \over 2}\) U \right] U d\bx \\
&\ - {1 \over p+1} \int_{B^n(0,\rho)} x_i \pa_if U^{p+1} d\bx + \({N-1 \over p+1} - {N-2 \over 2}\) \int_{B^n(0,\rho)} f U^{p+1} d\bx
\end{aligned}
\end{equation}
for all $\rho \in (0, \rho_1)$.
\end{lemma}

\subsection{Positive mass theorem}
In \cite{ABd}, Almaraz et al. introduced the mass of $N$-dimensional asymptotically flat manifolds with non-compact boundary
and proved the associated positive mass theorem for arbitrary manifolds of dimension $3 \le N \le 7$ and spin manifolds of dimension $N \ge 3$.
In \cite{AdQW}, Almaraz et al. used the positive mass theorem to describe the asymptotic behavior of the Green's function of the conformal Laplacian
on a smooth compact Riemannian manifold $(M, g)$ with boundary in terms of the mass.

The version of the positive mass theorem which we will apply in this paper is summarized in the following lemma.
This is a combination of Theorem 1.3 of \cite{ABd} and Proposition 3.5 of \cite{AdQW}.
\begin{lemma}\label{lemma_pmt}
For $3 \le N \le 7$, let $(M,g)$ be an $N$-dimensional smooth compact Riemannian manifold with boundary, and $y_0$ be an arbitrarily fixed point on $M$.
Suppose that we have the metric expansion
\begin{equation}\label{eq_g_ab}
g_{ab}(x) = \delta_{ab} + A_{ab}(x) + O(|x|^{2d+2}), \quad d = \left\lfloor \frac{N-2}{2} \right\rfloor
\end{equation}
with
\begin{equation}\label{eq_A_ab}
A_{iN}(x) = A_{NN}(x) = 0,\quad A_{ij}(x) = O(|x|^{d+1}),\quad \textnormal{trace}(A(x)) = O(|x|^{2d+2})
\end{equation}
in Fermi coordinates centered at $y_0$. Assume also that $G$ is a smooth positive function on $M \setminus \{y_0\}$ such that
\begin{equation}\label{eq_Green_ex}
G(x) = |x|^{2-N} + \phi(x)
\end{equation}
in the same coordinates, where $\phi$ is a smooth function on $M \setminus \{y_0\}$ satisfying
\begin{equation}\label{eq_phi}
\phi(x) = O(|x|^{d+3-N}|\log|x||) \quad \text{as } |x| \to 0.
\end{equation}
Then the manifold $(M \setminus \{y_0\}, G^{4 \over N-2}g)$ is asymptotically flat with the mass
\begin{equation}\label{eq_pmt}
m_0 = \lim_{\rho \to 0} \mci(y_0,\rho) \ge 0
\end{equation}
where
\begin{equation}\label{eq_mci}
\begin{aligned}
\mci(y_0,\rho) &= \frac{4(N-1)}{N-2} \int_{\pa_I B^N_+(0,\rho)} \(|x|^{2-N} \pa_a G(x) - \pa_a|x|^{2-N} G(x)\) \frac{x_a}{|x|}\, dS_x \\
&\ - \int_{\pa_I B^N_+(0,\rho)} \(\rho^{3-2N} x_a\pa_b A_{ab}(x) - 2N \rho^{1-2N} x_ax_bA_{ab}(x)\) dS_x.
\end{aligned}
\end{equation}
Furthermore, if $M$ is not conformally equivalent to the standard unit ball in $\R^N$,
\begin{equation}\label{eq_pmt_c}
R\big[G^{4 \over N-2}g\big] \ge 0 \quad \text{on } M \setminus \{y_0\} \quad \text{and} \quad
H\big[G^{4 \over N-2}g\big] \ge 0 \quad \text{on } \pa M \setminus \{y_0\},
\end{equation}
then $m_0 > 0$.
\end{lemma}
\noindent The integral expression $\mci$ for the mass was introduced by Brendle and Chen \cite{BC}.
In Lemma \ref{lemma_exp_2}, we will examine the relationship between the integral $\mci$ and the function $\mcp'$ defined in \eqref{eq_poho_0},
after choosing the function $G$ concretely.

\section{Basic properties of blow-up}\label{sec_blow}
\subsection{Characterization of blow-up points}\label{subsec_blow}
We recall the notion of blow-up, isolated blow-up and isolated simple blow-up.
By virtue of Proposition \ref{prop_blow_a}, it is enough to consider when the blow-up occurs near a point on the boundary.
The version we will use here is identical to those in \cite{Al, AdQW}.
\begin{defn}\label{def_blow}
Pick a small number $\rho_1 > 0$ such that $g_m$-Fermi coordinates centered at $y \in \pa M$
is well-defined in the closed geodesic half-ball $\overline{B^N_+(y,\rho_1)} \subset M$ for every $m \in \N$ and $y \in \pa M$.

\medskip \noindent (1) $y_0 \in \pa M$ is called a blow-up point of a sequence $\{U_m\}_{m \in \N}$ in $H^1(M)$ if there exists a sequence of points $\{y_m\}_{m \in \N} \subset \pa M$
such that $y_m$ is a local maximum of $U_m|_{\pa M}$ satisfying that $U_m(y_m) \to \infty$ and $y_m \to y_0$ as $m \to \infty$.
For the sake of brevity, we will often say that $y_m \to y_0$ is a blow-up point of $\{U_m\}_{m \in \N}$.

\medskip \noindent (2) $y_0 \in \pa M$ is an isolated blow-up point of $\{U_m\}_{m \in \N}$ if $y_0$ is a blow-up point such that
\[U_m(y) \le C d_{g_m}(y, y_m)^{-{1 \over p_m-1}} \quad \text{for any } y \in M \setminus \{y_m\},\ d_{g_m}(y, y_m) < \rho_2\]
for some $C > 0$ and $\rho_2 \in (0, \rho_1]$.

\medskip \noindent (3) Let $\ou_m$ be a weighted spherical average of $U_m$, i.e.,
\begin{equation}\label{eq_ou}
\ou_m(\rho) = \rho^{1 \over p_m-1} \({\int_{\pa_I B^N_+(y_m, \rho)} U_m\, dS_{g_m} \over \int_{\pa_I B^N_+(y_m, \rho)} dS_{g_m}}\), \quad \rho \in (0, \rho_1).
\end{equation}
We say that an isolated blow-up point $y_0$ of $\{U_m\}_{m \in \N}$ is simple if there exists a number $\rho_3 \in (0, \rho_2]$
such that $\ou_m$ possesses exactly one critical point in the interval $(0, \rho_3)$ for large $m \in \N$.
\end{defn}
Hereafter, we always assume that $U_m \in H^1(M)$ is a solution to \eqref{eq_Yamabe_3} with $p = p_m$, $g = g_m$ and $f_m = 1$ for each $m \in \N$.
For simplicity, we will just say that $\{U_m\}_{m \in \N}$ is a sequence of solutions to \eqref{eq_Yamabe_3}.
We also assume that $y_m \to y_0 \in \pa M$ is a blow-up point of $\{U_m\}_{m \in \N}$.
Set $M_m = U_m(y_m)$ and $\ep_m = M_m^{-(p_m-1)}$ for each $m \in \N$.
Obviously, $M_m \to \infty$ and $\ep_m \to 0$ as $m \to \infty$.

Choose a suitable positive smooth function $\omega_m$ on $M$ so that
the metric $\tig_m = \omega_m^{4 \over N-2} g_m$ on $M$ satisfies properties depicted in Lemma \ref{lemma_conf} where $y_*$ is replaced with $y_m$.
Then $\wtu_m = \omega_m^{-1} U_m$ is a solution to \eqref{eq_Yamabe_30} with $\tig = \tig_m$ and $\tif = \tif_m = \omega_m f_m$,
and a sequence $\{\tig_m\}_{m \in \N}$ of the metrics converges to a metric $\tig_0$ in $C^4(M, \R^{N \times N})$ as $m \to \infty$.
We shall often use $x \in \R^N_+$ to denote $\tig_m$-Fermi coordinates centered at $y_m$
so that $\wtu_m$ can be regarded as a function in $\R^N_+$ near the origin.

\subsection{Basic properties of blowing-up solutions}
Firstly, we study asymptotic behavior of a sequence $\{U_m\}_{m \in \N}$ of solutions to \eqref{eq_Yamabe_3} near blow-up points.
It can be proved as in e.g. Proposition 1.1 of \cite{HL} or Proposition 3.2 of \cite{FO}.
\begin{prop}\label{prop_blow_a}
Assume that $N \ge 3$ and $p \in [1+\vep_0, \frac{N}{N-2}]$.
Given arbitrary small $\vep_1 > 0$ and large $R > 0$, there are constants $C_0,\, C_1 > 0$ depending only on $(M^N,g), \vep_0,\, \vep_1$ and $R$ such that
if $U \in H^1(M)$ is a solution to \eqref{eq_Yamabe_2} with the property that $\max_M U \ge C_0$,
then $\frac{N}{N-2} - p < \vep_1$ and $U|_{\pa M}$ possesses local maxima $y_{01},\, \cdots y_{0\mcn} \in \pa M$ for some integer $\mcn = \mcn(U) \ge 1$, for which the following statements hold:

\medskip \noindent \textnormal{(1)} It is valid that
\[\overline{B_h(y_{0m_1}, \rho_{m_1})} \cap \overline{B_h(y_{0m_2}, \rho_{m_2})} = \emptyset \quad \text{for } 1 \le m_1 \ne m_2 \le \mcn\]
where $\rho_m = R U(y_{0m})^{-(p-1)}$.

\medskip \noindent \textnormal{(2)} For each $m = 1, \cdots, \mcn$, we have
\[\left\| U(y_{0m})^{-1} U\(U(y_{0m})^{-(p-1)} \cdot \) - W_{1,0} \right\|_{C^2(\overline{B^N_+(0,2R)})} \le \vep_1\]
in $g$-Fermi coordinates centered in $y_m$.

\medskip \noindent \textnormal{(3)} It holds that
\[U(y)\, d_h(y, \{y_{01}, \cdots, y_{0\mcn}\})^{1 \over p-1} \le C_1 \quad \text{for } y \in M.\]
\end{prop}

Secondly, we discuss behavior of a sequence of solutions $\{U_m\}_{m \in \N}$ to \eqref{eq_Yamabe_3} near isolated blow-up points.
The next lemma can be proved as in e.g. Proposition 1.4 of \cite{HL} or Lemma 2.6 of \cite{FO}.
\begin{lemma}\label{lemma_conv}
Let $y_m \to y_0 \in \pa M$ be an isolated blow-up point of a sequence $\{U_m\}_{m \in \N}$ of solutions to \eqref{eq_Yamabe_3}.
In addition, suppose that $\{R_m\}_{m \in \N}$ and $\{\tau_m\}_{m \in \N}$ are arbitrary sequences of positive numbers such that $R_m \to \infty$ and $\tau_m \to 0$ as $m \to \infty$.
Then $p_m \to \frac{N}{N-2}$ as $m \to \infty$, and $\{U_\ell\}_{\ell \in \N}$ and $\{p_\ell\}_{\ell \in \N}$
have subsequences $\{U_{\ell_m}\}_{m \in \N}$ and $\{p_{\ell_m}\}_{m \in \N}$ such that
\begin{equation}\label{eq_conv}
\left\| \ep_{\ell_m}^{1 \over p_{\ell_m}-1} U_{\ell_m}\(\ep_{\ell_m} \cdot\) - W_{1,0} \right\|_{C^2(\overline{B^N_+(0,R_m)})} \le \tau_m
\end{equation}
in $g_m$-Fermi coordinates centered in $y_m$ and $R_m \ep_{\ell_m} \to 0$ as $m \to \infty$.
\end{lemma}

\noindent Therefore, we can select $\{R_m\}_{m \in \N}$ and $\{U_{\ell_m}\}_{m \in \N}$ satisfying \eqref{eq_conv} and $R_m \ep_{\ell_m} \to 0$.
In order to simplify notations, we will use $\{U_m\}_{m \in \N}$ instead of $\{U_{\ell_m}\}_{m \in \N}$, and so on.

The following result is a simple consequence of Lemma \ref{lemma_conv} with the selection $\tau_m = \frac{1}{2} w_{1,0}(R_m)$.
Its proof is given in Corollary 3.6 of \cite{KMW}.
\begin{cor}\label{cor_i_blow}
Suppose that $y_m \to y_0 \in \pa M$ is an isolated blow-up point of a sequence $\{U_m\}_{m \in \N}$ of solutions to \eqref{eq_Yamabe_3}.

\medskip \noindent (1) If $\{\wtu_m\}_{m \in \N}$ is a sequence of solutions to \eqref{eq_Yamabe_30} constructed as in Subsection \ref{subsec_blow},
then $y_m \to y_0 \in \pa M$ is an isolated blow-up point of $\{\wtu_m\}_{m \in \N}$.

\medskip \noindent (2) The function $\ou_m$ in \eqref{eq_ou} has exactly one critical point in the interval $(0, R_m \ep_m)$ for large $m \in \N$.
In particular, if the isolated blow-up point $y_0 \in \pa M$ of $\{U_m\}_{m \in \N}$ is also simple, then $\ou_m'(r) < 0$ for all $r \in [R_m \ep_m, r_3)$; see Definition \ref{def_blow} (3).
\end{cor}

Thirdly, we examine how a sequence $\{U_m\}_{m \in \N}$ of solutions to \eqref{eq_Yamabe_3} behaves near isolated simple blow-up points.
See Proposition 4.3 of \cite{Al} for its proof.
\begin{prop}\label{prop_iso}
Assume that $N \ge 3$ and $y_m \to y_0 \in \pa M$ is an isolated simple blow-up point of a sequence $\{U_m\}_{m \in \N}$ of solutions to \eqref{eq_Yamabe_3},
and $\{\wtu_m\}_{m \in \N}$ is a sequence of solutions to \eqref{eq_Yamabe_30} constructed as in Subsection \ref{subsec_blow}.
Then there exists $C > 0$ and $\rho_4 \in (0,\rho_3)$ independent of $m \in \N$ such that
\begin{equation}\label{eq_U_m_est}
M_m \left| \nabla^{\ell} \wtu_m(x) \right| \le C|x|^{-(N-2+\ell)}
\quad \text{in } \left\{x \in \R^N_+: 0 < |x| \le \rho_4\right\}
\end{equation}
for $\ell = 0, 1, 2$ and
\[M_m \wtu_m(x) \ge C^{-1} G_m(x) \quad \text{in } \left\{x \in \R^N_+: R_m\ep_m \le |x| \le \rho_4\right\}\]
in $\tig_m$-Fermi coordinate system centered at $y_m$.
Here, $G_m$ is the Green's function satisfying
\[\begin{cases}
L_{g_m} G_m = 0 &\text{in } B^N_+(0,\rho_4), \\
B_{g_m} G_m = \delta_0 &\text{on } B^n(0,\rho_4), \\
G_m = 0 &\text{on } \pa_I B^N_+(0,\rho_4), \\
\lim_{|x| \to 0} |x|^{N-2} G_m(x) = 1,
\end{cases}\]
and $\delta_0$ is the Dirac measure centered at $0 \in \R^N_+$. Also,
\begin{equation}\label{eq_M_m_est}
M_m^{\delta_m} = M_m^{{N \over N-2}-p_m} \to 1 \quad \text{as } m \to \infty.
\end{equation}
\end{prop}

\section{Linear problems and refined blow-up analysis}\label{sec_lin}
\subsection{Linear problems}
In this subsection, we study the linear problem
\begin{equation}\label{eq_lin}
\begin{cases}
- \Delta \Psi = 2 \ep \pi_{ij} x_N \pa_{ij} W_{1,0} &\text{in } \R^N_+ = \R^n \times (0,\infty),\\
- \lim\limits_{x_N \to 0} \dfrac{\pa \Psi}{\pa x_N} = N w_{1,0}^{2 \over N-2} \Psi &\text{on } \R^n.
\end{cases}
\end{equation}
which arises from the first-order expansion of the metric on $M$; see Lemma \ref{lemma_metric}.
Here, $\ep > 0$ is a small parameter, $W_{1,0}$ is the function defined in \eqref{eq_W_lx}, $w_{1,0}(\bx) = W_{1,0}(\bx, 0)$ for $\bx \in \R^n$,
and $\pi$ is a trace-free symmetric $2$-tensor (that is, $n \times n$-matrices).

\begin{prop}\label{prop_lin}
Suppose that $N \ge 3$. There exists a smooth solution $\Psi$ to \eqref{eq_lin} and a constant $C > 0$ depending only on $N$ such that
\begin{equation}\label{eq_lin_1}
\left| \nabla^{\ell} \Psi(x) \right| \le C \ep \(\max\limits_{i, j = 1, \cdots, n} |\pi_{ij}|\) {1 \over 1+|x|^{N-3+\ell}} \quad \text{in } \R^N_+
\end{equation}
for any $\ell \in \N \cup \{0\}$,
\begin{equation}\label{eq_lin_2}
\Psi(0) = {\pa \Psi \over \pa x_1}(0) = \cdots = {\pa \Psi \over \pa x_n}(0) = 0
\quad \text{and} \quad
\int_{\R^n} w_{1,0}^{N \over N-2} \Psi d\bx = 0.
\end{equation}
\end{prop}
\begin{proof}
Pick a smooth function $\chi: [0,\infty) \to [0,1]$ such that $\chi(t) = 1$ on $[0,1]$ and 0 in $[2, \infty)$.
Set also $\chi_{\Lambda}(t) = \chi(\frac{t}{\Lambda})$ for any $\Lambda > 0$.
In Proposition 5.1 of \cite{Al}, it was proved that for each $\Lambda > 0$, there exists a smooth function $\Psi_{\Lambda}$ to
\begin{equation}\label{eq_lin_lam}
\begin{cases}
- \Delta \Psi = 2 \ep \pi_{ij} \chi_{\Lambda}(|x|) x_N \pa_{ij} W_{1,0} &\text{in } \R^N_+,\\
- \lim\limits_{x_N \to 0} \dfrac{\pa \Psi}{\pa x_N} = N w_{1,0}^{2 \over N-2} \Psi &\text{on } \R^n
\end{cases}
\end{equation}
satisfying \eqref{eq_lin_1}-\eqref{eq_lin_2} for some constant $C > 0$ depending only on $N$ (thereby being independent of $\Lambda > 0$).

Now, we choose a sequence $\{\Lambda_m\}_{m \in \N}$ of positive increasing numbers which diverges to $\infty$.
By the standard elliptic estimates, we may assume that the sequence $\{\Psi_{\Lambda_m}\}_{m \in \N}$ of solutions
to \eqref{eq_lin_lam} with $\Lambda = \Lambda_m$ converges to a smooth solution $\Psi$ to \eqref{eq_lin} in $C^2_{\text{loc}}(\overline{\R^N_+})$.
In particular, $\Psi$ satisfies \eqref{eq_lin_1}-\eqref{eq_lin_2}.
\end{proof}
\begin{rmk}
If $N \ge 5$, we infer from \eqref{eq_lin_1} that $\Psi \in D^{1,2}(\R^N_+)$.
In this case, one can argue as in Proposition 4.1 of \cite{KMW} to deduce the above proposition.
Also, \eqref{eq_lin}, \eqref{eq_lin_2} and the condition $\text{trace}(\pi) = 0$ imply
\[\int_{\R^N_+} \nabla \Psi \cdot \nabla W_{1,0} dx = 0.\]
\end{rmk}

For a better understanding of the function $\Psi$, we decompose it into two pieces:
The first part $\Phi$ is a rational function with parameters $a_1,\, a_2 \in \R$ whose Laplacian is the same as that of $\Psi$ in $\R^N_+$,
whose precise form is given in Lemma \ref{lemma_Phi}.
The second part $\Xi$ is a harmonic function with prescribed boundary condition, which is described in Lemma \ref{lemma_U}.
The proof of the lemmas are postponed until Appendix \ref{sec_app_PU}.
\begin{lemma}\label{lemma_Phi}
Suppose that $N \ge 4$. Given any $a_1, a_2 \in \R$, let
\begin{multline}\label{eq_Phi_n1}
\Phi(x) = \frac{\ep \pi_{ij} x_ix_j}{(|\bx|^2 + (x_N+1)^2)^{N \over 2}} \left[ \(\frac{N-2}{2}\) (x_N-1) \right. \\
\left. + \frac{a_1 (x_N+1)}{(|\bx|^2 + (x_N+1)^2)^2} + \frac{a_2}{|\bx|^2 + (x_N+1)^2} \right]
\end{multline}
in $\R^N_+$. Then it is a solution of
\begin{equation}\label{eq_Phi_n}
-\Delta \Phi = 2\ep \pi_{ij} x_N \pa_{ij} W_{1,0} \quad \text{in } \R^N_+.
\end{equation}
\end{lemma}
\begin{rmk}
The function $\Phi$ in \eqref{eq_Phi_n1} and the correction term $\psi_{\ep}$ defined in Page 387 of Marques \cite{Ma3} share a similar pointwise behavior.
However, $\Phi$ have two degrees of freedom on the coefficients, while $\psi_{\ep}$ has only one.
\end{rmk}

\begin{lemma}\label{lemma_U}
Suppose that $N \ge 4$. The function $\Xi = \Psi - \Phi$ satisfies
\begin{equation}\label{eq_U}
\begin{cases}
- \Delta \Xi = 0 &\text{in } \R^N_+,\\
-\lim\limits_{x_N \to 0} \dfrac{\pa \Xi}{\pa x_N} = N w_{1,0}^{2 \over N-2} \Xi + q &\text{on } \R^n
\end{cases}
\end{equation}
where
\begin{equation}\label{eq_q}
q(\bx) = \frac{\ep \pi_{ij}x_ix_j}{(|\bx|^2+1)^{N \over 2}} \left[ \frac{N-2}{2} + a_1 \left\{ \frac{1}{(|\bx|^2+1)^2} - \frac{4}{(|\bx|^2+1)^3} \right\} - \frac{2a_2}{(|\bx|^2+1)^2} \right]
\end{equation}
on $\R^n$.
\end{lemma}

We prove an auxiliary lemma that comes from the mountain pass structure of the fractional Yamabe problem in $\R^N_+$.
It will used in the proof of Proposition \ref{prop_van} for $N = 5$ and $6$.
\begin{lemma}
For $N \ge 5$, it holds that $\Xi \in D^{1,2}(\R^N_+)$ and
\begin{equation}\label{eq_mp}
\int_{\R^N_+} |\nabla \Xi|^2 dx - N \int_{\R^n} w_{1,0}^{2 \over N-2} \Xi^2 d\bx \ge 0.
\end{equation}
\end{lemma}
\begin{proof}
By \eqref{eq_lin_1} and \eqref{eq_Phi_n1}, we readily observe that $\Xi \in D^{1,2}(\R^N_+)$.

Testing $\Xi$ in \eqref{eq_W_eq} and $W_{1,0}$ in \eqref{eq_U} gives
\begin{align*}
(N-2) \int_{\R^n} w_{1,0}^{N \over N-2} \Xi d\bx
&= \int_{\R^N_+} \nabla \Xi \cdot \nabla W_{1,0} dx \\
&= N\int_{\R^n} w_{1,0}^{N \over N-2} \Xi d\bx + \int_{\R^n} q w_{1,0} d\bx
= N\int_{\R^n} w_{1,0}^{N \over N-2} \Xi d\bx
\end{align*}
where the last equality holds owing to the condition that $\text{trace}(\pi) = 0$. Thus
\begin{equation}\label{eq_ortho}
\int_{\R^N_+} \nabla \Xi \cdot \nabla W_{1,0} dx = \int_{\R^n} w_{1,0}^{N \over N-2} \Xi d\bx = 0.
\end{equation}
One can now argue as in the proof of Lemma 4.5 of \cite{DKP} to deduce the validity of \eqref{eq_mp}. Here we provide a more direct proof.

Define the energy functional $J$ of \eqref{eq_W_eq} as
\[J(U) = \frac{1}{2} \int_{\R^N_+} |\nabla U|^2 dx - \frac{(N-2)^2}{2(N-1)} \int_{\R^n} U_+^{2(N-1) \over N-2} d\bx \quad \text{for } U \in D^{1,2}(\R^N_+)\]
and the Nehari manifold $\mcm$ associated with $J$ as
\[\mcm = \left\{U \in D^{1,2}(\R^N_+) \setminus \{0\}: \int_{\R^N_+} |\nabla U|^2 dx = (N-2) \int_{\R^n} U_+^{2(N-1) \over N-2} d\bx \right\}\]
where $U_+ = \max\{U,0\}$. Then $J$ is a functional of class $C^2$, $\mcm$ is a $C^1$-Hilbert manifold and $W_{1,0} \in \mcm$.
Moreover, the tangent space $T_{W_{1,0}} \mcm$ of $\mcm$ at $W_{1,0}$ is
\[T_{W_{1,0}} \mcm = \left\{U \in D^{1,2}(\R^N_+): \int_{\R^N_+} \nabla W_{1,0} \cdot \nabla U dx = (N-1) \int_{\R^n} w_{1,0}^{N \over N-2} U d\bx \right\}.\]
In particular, \eqref{eq_ortho} implies that $\Xi \in T_{W_{1,0}} \mcm$.
By Theorem 1.1 of \cite{Es5}, $W_{1,0}$ is a minimizer of $J$ in $\mcm$. Therefore
\[0 \le \left. \frac{d^2J(W_{1,0}+\vep \Xi)}{d\vep^2} \right|_{\vep = 0} = \int_{\R^N_+} |\nabla \Xi|^2 dx - N \int_{\R^n} w_{1,0}^{2 \over N-2} \Xi^2 d\bx,\]
which is \eqref{eq_mp}.
\end{proof}

\subsection{Refined blow-up analysis}
By using Proposition \ref{prop_lin}, we can analyze the $\ep_m$-order behavior of a sequence $\{U_m\}_{m \in \N}$ of solutions to \eqref{eq_Yamabe_3} near isolated simple blow-up points.
Owing to Corollary \ref{cor_i_blow} (i) and Lemma \ref{lemma_conf}, $y_m \to y_0$ is an isolated blow-up point of
a sequence $\{\wtu_m\}_{m \in \N}$ of solutions to \eqref{eq_Yamabe_30} constructed in Subsection \ref{subsec_blow}, and $M_m = \wtu_m(y_m)$.
\begin{prop}
Suppose that $N \ge 4$ and $y_m \to y_0 \in \pa M$ is an isolated simple blow-up point of $\{U_m\}_{m \in \N}$.
Let $\Psi_m$ be the solution of \eqref{eq_lin} with $\ep = \ep_m$ and $\pi = \pi[\tig_m](y_m)$,
and
\begin{equation}\label{eq_wtv}
\wtv_m(x) = \ep_m^{1 \over p_m-1} \wtu_m\(\ep_m x\) \quad \text{in } B^N_+(0, \rho_4 \ep_m^{-1}).
\end{equation}
Then there exists $C > 0$ and $\rho_5 \in (0, \rho_4]$ independent of $m \in \N$ such that
\begin{equation}\label{eq_V_m_est}
\left| \nabla_{\bx}^{\ell} \wtv_m - \nabla_{\bx}^{\ell}(W_{1,0} + \Psi_m) \right|(x) \le {C \ep_m^2 \over 1+|x|^{N-4+\ell}}
\quad \text{in } B^N_+(0,\rho_5 \ep_m^{-1})
\end{equation}
for $\ell = 0, 1, 2$.
\end{prop}
\noindent For $N \ge 5$, the proposition was proved in Proposition 6.1 of \cite{Al} and Proposition 4.2 of \cite{KMW}.
Also, a slight modification of the arguments in \cite{Al, KMW} shows that it also holds for $N = 4$.
Check Proposition 5.3 of \cite{AdQW} where its 3-dimensional version was derived.

\section{Quantitative analysis on the trace-free second fundamental form}\label{sec_van}
\subsection{Vanishing theorem of the trace-free second fundamental form}
In the next proposition, we prove that the trace-free second fundamental form must vanish at each isolated simple blow-up point of blowing-up solutions when $N = 4, 5, 6$.
An analogous result for $N \ge 7$ can be found in Theorem 7.1 of \cite{Al}.
\begin{prop}\label{prop_van}
Suppose that $N = 4, 5, 6$ and $y_m \to y_0 \in \pa M$ is an isolated simple blow-up point of the sequence $\{U_m\}_{m \in \N}$ of the solutions to \eqref{eq_Yamabe_3}.
If $\{\tig_m\}_{m \in \N}$ is a sequence of the metrics constructed in Subsection \ref{subsec_blow},
then there exists $C > 0$ independent of $m \in \N$ such that
\begin{equation}\label{eq_van}
\|\pi[\tig_m](y_m)\|^2 \le \begin{cases}
\dfrac{C}{|\log \ep_m|} &\text{for } N = 4,\\
C \ep_m|\log \ep_m| &\text{for } N = 5,\\
C \ep_m &\text{for } N = 6.
\end{cases}
\end{equation}
Particularly, $\pi[\tig_0](y_0) = 0$.
\end{prop}

Let $\{\wtu_m\}_{m \in \N}$ be a sequence of solutions to \eqref{eq_Yamabe_30} depicted in Subsection \ref{subsec_blow}.
By appealing $\tig_m$-Fermi coordinates on $M$ centered at $y_m$, we regard $\wtu_m$ as a function defined near $0 \in \R^N_+$.
For brevity, we write $\pi_m = \pi[\tig_m](y_m)$ for all $m \in \N$.

\medskip
Denoting $\hg_m = \tig_m(\ep_m \cdot)$ and $\hf_m = \tif_m(\ep_m \cdot)$,
we see from \eqref{eq_Yamabe_30} that the function $\wtv_m$ introduced in \eqref{eq_wtv} solves
\[\begin{cases}
- \Delta \wtv_m = - \left[\dfrac{N-2}{4(N-1)}\right] \ep_m^2 R[\tig_m](\ep_m \cdot) \wtv_m + (\Delta_{\hg_m} - \Delta) \wtv_m &\text{in } B^N_+(0,\rho_5 \ep_m^{-1}),\\
- \dfrac{\pa \wtv_m}{\pa x_N} + \left[\dfrac{N-2}{2}\right] \ep_m H[\tig_m](\ep_m \cdot) \wtv_m = (N-2) \hf_m^{-\delta_m} \wtv_m^{p_m} &\text{on } B^n(0,\rho_5 \ep_m^{-1}).
\end{cases}\]
Thus, employing Pohozaev's identity \eqref{eq_poho_2}, one can write
\begin{equation}\label{eq_poho_3}
\mcp\(\wtv_m, \rho\ep_m^{-1}\) = \mcp_{1m}\(\wtv_m, \rho\ep_m^{-1}\) + {\delta_m \over p_m+1} \mcp_{2m}\(\wtv_m, \rho\ep_m^{-1}\) \quad \text{for any } \rho \in (0, \rho_5]
\end{equation}
where $\mcp$ is the function defined in \eqref{eq_poho_1} with $f = (N-2)\hf_m^{-\delta_m}$,
\begin{align}
&\ \mcp_{1m}(U, \rho) \nonumber \\
&= \int_{B^N_+(0,\rho)} \left[\left\{ \dfrac{N-2}{4(N-1)}\right\} \ep_m^2 R[\tig_m](\ep_m \cdot) U + (\Delta - \Delta_{\hg_m})U \right] \cdot \left[x_a \pa_a U + \({N-2 \over 2}\) U \right] dx \nonumber \\
&\ + \(\frac{N-2}{2}\) \ep_m \int_{B^n(0,\rho)} H[\tig_m](\ep_m \cdot) \left[x_i \pa_i U + \({N-2 \over 2}\) U \right] U d\bx \label{eq_mcp_1}
\end{align}
and
\[\mcp_{2m}(U, \rho) = - \int_{B^n(0,\rho)} x_i \pa_i\hf_m \hf_m^{-(\delta_m+1)} U^{p_m+1} d\bx + \(\frac{N-2}{2}\) \int_{B^n(0,\rho)} \hf_m^{-\delta_m} U^{p_m+1} d\bx.\]

\medskip
The left-hand side of \eqref{eq_poho_3} involves with the boundary integrals only. By \eqref{eq_U_m_est}, \eqref{eq_M_m_est} and \eqref{eq_wtv}, 
there exists a constant $C > 0$ independent of $m \in \N$ and $\rho \in (0,\rho_5]$ such that
\begin{equation}\label{eq_poho_31}
\mcp\(\wtv_m, \rho\ep_m^{-1}\) = O(\ep_m^{N-2}).
\end{equation}

The right-hand side of \eqref{eq_poho_3} involves with the interior integrals.
We can take $\rho$ so small that
\begin{equation}\label{eq_poho_32}
\mcp_{2m}\(\wtv_m, \rho\ep_m^{-1}\) \ge 0.
\end{equation}
Also, choosing $\kappa \ge 2$ in \eqref{eq_kappa}, we may assume that the second integral in the right-hand side of \eqref{eq_mcp_1} is bounded by
\begin{align*}
&\ \ep_m \int_{B^n(0,\rho\ep_m^{-1})} \left| H[\tig_m](\ep_m \bx)\right| \left|x_i \pa_i \wtv_m + \({N-2 \over 2}\) \wtv_m \right| \left|\wtv_m\right| d\bx \\
&\le C \ep_m^{\kappa+1} \int_{B^n(0,\rho\ep_m^{-1})} \frac{|\bx|^{\kappa}}{1+|\bx|^{2(N-2)}} d\bx
= O(\ep_m^3) + O(\ep_m^{N-2});
\end{align*}
see the derivation of \eqref{eq_H} below. Hence, by fixing $\rho$ small enough and invoking \eqref{eq_V_m_est}, we get
\begin{multline}\label{eq_mt}
\mcp_{1m}\(\wtv_m, \rho\ep_m^{-1}\) \\
= F_m(W_{1,0}, W_{1,0}) + \left[ F_m(W_{1,0}, \Psi_m) + F_m(\Psi_m, W_{1,0}) \right] + \begin{cases}
O(\ep_m^2) &\text{for } N = 4,\\
O(\ep_m^3 |\log \ep_m|) &\text{for } N = 5,\\
O(\ep_m^3) &\text{for } N \ge 6
\end{cases}
\end{multline}
where
\begin{multline}\label{eq_Fm}
F_m(V_1, V_2) = \int_{B^N_+(0,\rho\ep_m^{-1})} \left[\left\{{N-2 \over 4(N-1)}\right\} \ep_m^2 R[\tig_m](\ep_m x) V_1 + (\Delta-\Delta_{\hg_m}) V_1 \right] \\
\times \left[ x \cdot \nabla V_2 + \({N-2 \over 2}\) V_2 \right] dx
\end{multline}
and $\Psi_m$ is the solution of \eqref{eq_lin} with $\ep = \ep_m$ and $\pi = \pi_m$.
To estimate \eqref{eq_mt}, we divide the cases according to the dimension $N$.
We examine the case $N = 5$ first, $N = 6$ second, and $N = 4$ at last.

\medskip \noindent
\textsc{Case $N = 5$:} By putting $n = 4$ and $\gamma = \frac{1}{2}$ in (5.9) of \cite{KMW}, one can compute that
\begin{equation}\label{eq_FW}
F_m(W_{1,0}, W_{1,0}) = C_1 \ep_m^2 \|\pi_m\|^2 + O(\ep_m^3 |\log \ep_m|)
\end{equation}
where
\begin{align*}
C_1 = -\frac{1}{8} \int_{\R^5_+} x_5^2 |\nabla_{\bx} W_{1,0}|^2 dx
&= - \frac{9}{8} \left|\S^3\right| \int_0^{\infty} \frac{x_5^2 dx_5}{(x_5+1)^4} \int_0^{\infty} \frac{t^5 dt}{(t^2+1)^5} \\
&= - \frac{9}{8} \left|\S^3\right| \cdot \frac{1}{3} \cdot \frac{1}{24} = - \frac{1}{64} \left|\S^3\right|.
\end{align*}
Besides, it was shown in (5.10) of \cite{KMW} that
\[F_m(W_{1,0}, \Psi_m) + F_m(\Psi_m, W_{1,0}) \ge O(\ep_m^3 |\log \ep_m|).\]
However, it is not enough to deduce the proposition because $C_1 < 0$. We will improve the estimate in the next result.
\begin{lemma}\label{lemma_FW}
It holds that
\begin{align}
&\ F_m(W_{1,0}, \Psi_m) + F_m(\Psi_m, W_{1,0}) \label{eq_FW_2}\\
&\ge \left|\S^3\right| \( -\frac{1}{128} + \frac{a_1}{480} - \frac{11 a_1^2}{60480}
+ \frac{a_2}{160} - \frac{a_1 a_2}{1680} - \frac{a_2^2}{1680}\) \ep_m^2 \|\pi_m\|^2 + O(\ep_m^3 |\log \ep_m| \cdot \|\pi_m\|). \nonumber
\end{align}
\end{lemma}
\begin{proof}[Proof of Lemma \ref{lemma_FW}]
We see from Derivation of (5.10) of \cite{KMW} that
\begin{equation} \label{eq_mt_22}
\begin{aligned}
&\ F_m(W_{1,0}, \Psi_m) + F_m(\Psi_m, W_{1,0}) \\
&= -2\ep_m (\pi_m)_{ij} \left[ \int_{\R^5_+} x_5 \pa_{ij}W_{1,0} \(x \cdot \nabla \Psi_m + {3 \over 2} \Psi_m\) dx + \int_{\R^5_+} x_5 \pa_{ij}\Psi_m Z_{1,0}^0 dx \right] \\
&\ + O(\ep_m^3 |\log \ep_m| \cdot \|\pi_m\|) \\
&= - 2\ep_m (\pi_m)_{ij} \int_{\R^5_+} x_5 \pa_i W_{1,0} \pa_j \Psi_m dx + O(\ep_m^3 |\log \ep_m| \cdot \|\pi_m\|) \\
&= 2\ep_m (\pi_m)_{ij} \( \int_{\R^5_+} x_5 \pa_{ij} W_{1,0} \Phi_m dx + \int_{\R^5_+} x_5 \pa_{ij} W_{1,0} \Xi_m dx \) + O(\ep_m^3 |\log \ep_m| \cdot \|\pi_m\|)
\end{aligned}
\end{equation}
where $\Phi_m$ and $\Xi_m$ are defined by \eqref{eq_Phi_n1} and \eqref{eq_U} with $\ep = \ep_m$ and $\pi = \pi_m$, and so $\Psi_m = \Phi_m + \Xi_m$.

On the other hand, by testing $\Xi_m$ in \eqref{eq_lin}, we obtain
\begin{align*}
&\ 2 \ep_m (\pi_m)_{ij} \int_{\R^5_+} x_5 \pa_{ij} W_{1,0} \Xi_m dx \\
&= \int_{\R^5_+} \nabla \Psi_m \cdot \nabla \Xi_m dx - 5\int_{\R^4} w_{1,0}^{2 \over 3} \Psi_m \Xi_m d\bx \\
&= \int_{\R^5_+} \nabla \Phi_m \cdot \nabla \Xi_m dx - 5\int_{\R^4} w_{1,0}^{2 \over 3} \Phi_m \Xi_m d\bx
+ \int_{\R^5_+} |\nabla \Xi_m|^2 dx - 5\int_{\R^4} w_{1,0}^{2 \over 3} \Xi_m^2 d\bx.
\end{align*}
Testing $\Phi_m$ in \eqref{eq_U}, we find
\[\int_{\R^5_+} \nabla \Xi_m \cdot \nabla \Phi_m dx = 5\int_{\R^4} w_{1,0}^{2 \over 3} \Xi_m \Phi_m d\bx + \int_{\R^4} q_m \Phi_m d\bx\]
where $q_m$ is the function defined by \eqref{eq_q} with $\ep = \ep_m$ and $\pi = \pi_m$.
Thus it follows from \eqref{eq_mp} that
\begin{equation}\label{eq_mt_23}
2\ep_m (\pi_m)_{ij} \int_{\R^5_+} x_5 \pa_{ij} W_{1,0} \Xi_m dx
\ge \int_{\R^4} q_m \Phi_m d\bx.
\end{equation}

Combining \eqref{eq_mt_22} and \eqref{eq_mt_23}, we obtain
\begin{multline}\label{eq_FW_3}
F_m(W_{1,0}, \Psi_m) + F_m(\Psi_m, W_{1,0}) \\
\ge (\pi_m)_{ij} \( 2\ep_m \int_{\R^5_+} x_5 \pa_{ij} W_{1,0} \Phi_m dx + \int_{\R^4} q_m \Phi_m d\bx \) + O(\ep_m^3 |\log \ep_m| \cdot \|\pi_m\|).
\end{multline}

By applying \eqref{eq_Phi_n1}, we evaluate
\begin{align}
&\ 2\ep_m (\pi_m)_{ij} \int_{\R^5_+} x_5 \pa_{ij} W_{1,0} \Phi_m dx \nonumber \\
&= 30\ep_m (\pi_m)_{ij} \int_{\R^5_+} \frac{x_ix_jx_5}{(|\bx|^2 + (x_5+1)^2)^{7 \over 2}} \Phi_m dx \nonumber \\
&= 30 \ep_m^2 (\pi_m)_{ij} (\pi_m)_{kl} \int_{\R^5_+} x_ix_jx_kx_l \left[ \frac{3}{2} \cdot \frac{x_5(x_5-1)}{(|\bx|^2+(x_5+1)^2)^6} \right. \nonumber \\
&\hspace{150pt} \left. + a_1 \frac{x_5(x_5+1)}{(|\bx|^2+(x_5+1)^2)^8} + a_2 \frac{x_5}{(|\bx|^2+(x_5+1)^2)^7} \right] dx \label{eq_FW_31} \\
&= \left|\S^3\right| \left[ \frac{15}{4} \int_0^{\infty} \frac{x_5(x_5-1) dx_5}{(x_5+1)^4} \int_0^{\infty} \frac{t^7 dt}{(t^2+1)^6}
+ \frac{5a_1}{2} \int_0^{\infty} \frac{x_5 dx_5}{(x_5+1)^7} \int_0^{\infty} \frac{t^7 dt}{(t^2+1)^8} \right. \nonumber \\
&\hspace{178pt} \left. + \frac{5a_2}{2} \int_0^{\infty} \frac{x_5 dx_5}{(x_5+1)^6} \int_0^{\infty} \frac{t^7 dt}{(t^2+1)^7} \right] \ep_m^2 \|\pi_m\|^2 \nonumber \\
&= \left|\S^3\right| \(\frac{1}{64} + \frac{a_1}{3360} + \frac{a_2}{960}\) \ep_m^2 \|\pi_m\|^2 \nonumber
\end{align}
and
\begin{align}
&\ \int_{\R^4} q_m \Phi_m d\bx \nonumber \\
&= \ep_m^2 (\pi_m)_{ij} (\pi_m)_{kl} \int_{\R^4} \frac{x_ix_jx_kx_l}{(|\bx|^2+1)^5}
\left[ \frac{3}{2} + a_1 \left\{ \frac{1}{(|\bx|^2+1)^2}
- \frac{4}{(|\bx|^2+1)^3} \right\} - \frac{2a_2}{(|\bx|^2+1)^2} \right] \nonumber \\
&\hspace{136pt} \times \left[ - \frac{3}{2} + \frac{a_1}{(|\bx|^2 + 1)^2}
+ \frac{a_2}{|\bx|^2 + 1} \right] d\bx \nonumber \\
&= \frac{1}{12} \left|\S^3\right| \int_0^{\infty} \frac{r^7}{(r^2+1)^5}
\left[ \frac{3}{2} + a_1 \left\{ \frac{1}{(r^2+1)^2}
- \frac{4}{(r^2+1)^3} \right\} - \frac{2a_2}{(r^2+1)^2} \right] \label{eq_FW_32} \\
&\hspace{101pt} \times \left[ - \frac{3}{2} + \frac{a_1}{(r^2 + 1)^2}
+ \frac{a_2}{r^2 + 1} \right] d\bx \cdot \ep_m^2 \|\pi_m\|^2 \nonumber \\
&= \left|\S^3\right| \(-\frac{3}{128} + \frac{a_1}{560} - \frac{11 a_1^2}{60480} + \frac{a_2}{192} - \frac{a_1 a_2}{1680} - \frac{a_2^2}{1680} \) \ep_m^2 \|\pi_m\|^2. \nonumber
\end{align}
Putting \eqref{eq_FW_3}-\eqref{eq_FW_32}, we deduce \eqref{eq_FW_2}.
\end{proof}

\begin{cor}\label{cor_FW}
It holds that
\begin{multline}\label{eq_FW_4}
F_m(W_{1,0}, W_{1,0}) + \left[ F_m(W_{1,0}, \Psi_m) + F_m(\Psi_m, W_{1,0}) \right] \\
\ge \frac{3}{2560} \left|\S^3\right| \ep_m^2 \|\pi_m\|^2 + O(\ep_m^3 |\log \ep_m|) + O(\ep_m^3 |\log \ep_m| \cdot \|\pi_m\|).
\end{multline}
\end{cor}
\begin{proof}
Thus we conclude from \eqref{eq_FW} and \eqref{eq_FW_2} that
\begin{multline*}
F_m(W_{1,0}, W_{1,0}) + \left[ F_m(W_{1,0}, \Psi_m) + F_m(\Psi_m, W_{1,0}) \right] \\
\ge \left|\S^3\right| P(a_1,a_2) \ep_m^2 \|\pi_m\|^2 + O(\ep_m^3 |\log \ep_m|) + O(\ep_m^3 |\log \ep_m| \cdot \|\pi_m\|)
\end{multline*}
where
\[P(a_1,a_2) = -\frac{3}{128} + \frac{a_1}{480} - \frac{11 a_1^2}{60480} + \frac{a_2}{160} - \frac{a_1 a_2}{1680} - \frac{a_2^2}{1680}.\]
It holds that
\[\max_{a_1, a_2 \in \R} P(a_1,a_2) = P\(-\frac{63}{4}, \frac{105}{8}\) = \frac{3}{2560}.\]
Hence the assertion follows.
\end{proof}

\begin{proof}[Completion of the proof of Proposition \ref{prop_van} for $N = 5$]
Because $\tig_m \to \tig_0$ in $C^4(M, \R^{N \times N})$ as $m \to \infty$, the values of $\|\pi_m\|$ are uniformly bounded in $m \in \N$.
From \eqref{eq_poho_3}, \eqref{eq_poho_31}, \eqref{eq_poho_32} and \eqref{eq_FW_4}, we discover
\[O(\ep_m^3) \ge \frac{3}{2560} \left|\S^3\right| \ep_m^2 \|\pi_m\|^2 + O(\ep_m^3 |\log \ep_m|).\]
Accordingly,
\[O(\ep_m) \ge \frac{3}{2560} \left|\S^3\right| \|\pi_m\|^2 + O(\ep_m |\log \ep_m|).\]
Taking $m \to \infty$ on the both sides, we get \eqref{eq_van} for $N = 5$.
\end{proof}

\medskip \noindent
\textsc{Case $N = 6$:} The strategy is the same as the case $N = 5$.
By inserting $n = 5$ and $\gamma = \frac{1}{2}$ in (5.9) of \cite{KMW}, one can compute that
\[F_m(W_{1,0}, W_{1,0}) = O(\ep_m^3).\]
Also, computing as in Lemma \ref{lemma_FW}, we obtain
\begin{lemma}
It holds that
\begin{multline*}
F_m(W_{1,0}, \Psi_m) + F_m(\Psi_m, W_{1,0}) \\
\ge \left|\S^4\right| \( - \frac{\pi}{320} + \frac{a_1 \pi}{3584} - \frac{3 a_1^2 \pi}{163840} + \frac{a_2 \pi}{1280} - \frac{a_1 a_2 \pi}{16384} - \frac{a_2^2 \pi}{16384}\) \ep_m^2 \|\pi_m\|^2 + O(\ep_m^3).
\end{multline*}
\end{lemma}
Choosing the parameters $a_1 = -\frac{128}{7}$ and $a_2 = \frac{544}{35}$, we get
\begin{cor}
It holds that
\[F_m(W_{1,0}, W_{1,0}) + \left[ F_m(W_{1,0}, \Psi_m) + F_m(\Psi_m, W_{1,0}) \right]
\ge \frac{31\pi}{78400} \left|\S^3\right| \ep_m^2 \|\pi_m\|^2 + O(\ep_m^3).\]
\end{cor}
\noindent From this, the desired result \eqref{eq_van} for $N = 6$ follows.

\medskip \noindent
\textsc{Case $N = 4$:} Because of the integrability issue on $W_{1,0}$, the computation becomes a little bit trickier than before.
Especially, it turns out that the terms involving $a_1$ and $a_2$ contribute nothing.
This is because the integrals involving them are $O(\ep_m^2)$, while the main order of $\mcp_{1m}(\wtv_m, \rho\ep_m^{-1})$ is $\ep_m^2 |\log \ep_m|$. Hence we set $a_1 = a_2 = 0$.

\begin{lemma}\label{eq_nW}
It holds that
\begin{equation}\label{eq_nFW}
F_m(W_{1,0}, W_{1,0}) = -\frac{\pi}{24} \left|\S^2\right| \|\pi_m\|^2 \ep_m^2 \log (\rho\ep_m^{-1}) + O(\ep_m^2).
\end{equation}
\end{lemma}
\begin{proof}
Lemma \ref{lemma_conf} and the Gauss-Codazzi equation implies that \[R[\tig_m](\ep_m x) = - \|\pi_m\|^2 + O(\ep_m |x|) \quad \text{in } B^5_+(0,\rho\ep_m^{-1}).\]
From this, Lemma \ref{lemma_metric} (more precisely, Lemmas 3.1 and 3.2 of \cite{Es}) and \eqref{eq_Fm}, we find that
\begin{equation}\label{eq_nFW_11}
F_m(W_{1,0}, W_{1,0}) = \wtf_{0m} + \wtf_{1m} + \wtf_{2m} + O(\ep_m^2)
\end{equation}
where
\begin{align*}
\wtf_{0m} &= \frac{1}{6}\, \ep_m^2 \int_{B^4_+(0,\rho\ep_m^{-1})} R[\tig_m](\ep_m x) W_{1,0} Z_{1,0}^0 dx \\
&= - \frac{1}{6}\, \ep_m^2 \|\pi_m\|^2 \int_{B^4_+(0,\rho\ep_m^{-1})} W_{1,0} Z_{1,0}^0 dx + O(\ep_m^2),
\end{align*}
\begin{align*}
\wtf_{1m} &= \int_{B^4_+(0,\rho\ep_m^{-1})} (\delta^{ij}-\hg_m^{ij}) \pa_{ij} W_{1,0} Z_{1,0} dx \\
&= - \frac{1}{3}\, \ep_m^2 \left[3\|\pi_m\|^2 + R_{NN}[\tig_m](y_m) \right] \int_{B^4_+(0,\rho\ep_m^{-1})} x_4^2 \Delta_{\bx} W_{1,0} Z_{1,0}^0 dx + O(\ep_m^2) \\
&= -\frac{2}{3}\, \ep_m^2 \|\pi_m\|^2 \int_{B^4_+(0,\rho\ep_m^{-1})} x_4^2 \Delta_{\bx} W_{1,0} Z_{1,0}^0 dx + O(\ep_m^2)
\end{align*}
and
\begin{equation}\label{eq_nFW_14}
\begin{aligned}
\wtf_{2m} &= -\int_{B^4_+(0,\rho\ep_m^{-1})} \(\frac{\pa_a \sqrt{|\hg_m|}}{\sqrt{|\hg_m|}}\) \hg_m^{ab} \pa_b W_{1,0} Z_{1,0} dx \\
&= \ep_m^2 \left[\|\pi_m\|^2 + R_{NN}[\tig_m](y_m) \right] \int_{B^4_+(0,\rho\ep_m^{-1})} \hg_m^{ab} \pa_b W_{1,0} Z_{1,0} dx + O(\ep_m^2) = O(\ep_m^2).
\end{aligned}
\end{equation}

On the other hand, since
\begin{align*}
&\ \int_{B^4_+(0,\rho\ep_m^{-1})} W_{1,0} Z_{1,0}^0 dx \\
&= \int_0^{\rho\ep_m^{-1}} \int_{\R^3} \frac{1-|\bx|^2-x_4^2}{(|\bx|^2+(x_4+1)^2)^3} d\bx dx_4 + O(1) \\
&= - \left|\S^2\right| \left[ \int_0^{\rho\ep_m^{-1}} \frac{dx_4}{x_4+1}
\int_0^{\infty} \frac{t^4 dt}{(t^2+1)^3}
+ \int_0^{\rho\ep_m^{-1}} \frac{x_4^2 dx_4}{(x_4+1)^3}
\int_0^{\infty} \frac{t^2 dt}{(t^2+1)^3} \right] + O(1) \\
&= - \frac{\pi}{4} \left|\S^2\right| \log(\rho\ep_m^{-1}) + O(1),
\end{align*}
we have
\begin{equation}\label{eq_nFW_12}
\wtf_{0m} = {\pi \over 24} \left|\S^2\right| \|\pi_m\|^2 \ep_m^2 \log(\rho\ep_m^{-1}) + O(\ep_m^2).
\end{equation}
Moreover,
\begin{align*}
&\ \int_{B^4_+(0,\rho\ep_m^{-1})} x_4^2 \Delta_{\bx} W_{1,0} Z_{1,0}^0 dx
\\ &= 2 \int_0^{\rho\ep_m^{-1}} \int_{\R^3} \frac{ x_4^2 \left[|\bx|^2-3(x_4+1)^2\right] (1-|\bx|^2-x_4^2)}{(|\bx|^2+(x_4+1)^2)^5} dx \\
&= -2 \left|\S^2\right| \int_0^{\rho\ep_m^{-1}} \int_0^{\infty} \frac{ r^2 x_4^2 \left[r^2-3(x_4+1)^2\right] (r^2+x_4^2)}{(r^2+(x_4+1)^2)^5} dr dx_4 + O(1) \\
&= \frac{\pi}{8} \left|\S^2\right| \log(\rho\ep_m^{-1}) + O(1),
\end{align*}
from which we deduce that
\begin{equation}\label{eq_nFW_13}
\wtf_{1m} = - {\pi \over 12} \left|\S^2\right| \|\pi_m\|^2 \ep_m^2 \log(\rho\ep_m^{-1}) + O(\ep_m^2).
\end{equation}
Combining \eqref{eq_nFW_11}-\eqref{eq_nFW_13}, we obtain \eqref{eq_nFW}.
\end{proof}

Unlike the cases $N = 5$ and $6$, we do not exploit the mountain pass structure of the boundary Yamabe problem in $\R^N_+$.
Instead, we use the integrability (or the decay property) of the functions involving the problem.

We define
\begin{equation}\label{eq_Phi_delta}
\Phi_{\delta}(x) = \ep \pi_{ij}x_ix_j \left[ \frac{x_4-1}{(|\bx|^2 + (x_4+1)^2)^2} + \frac{\delta}{(|\bx|^2 + (x_4+1)^2)^{3 \over 2}} \right]
\end{equation}
for $\delta$ small, which resembles the modified correction term $\psi_{\ep,\delta}$ defined in Page 400 of \cite{Ma3}.
If $\delta = 0$, the function $\Phi_{\delta}$ is reduced to $\Phi$ in \eqref{eq_Phi_n1} with $a_1 = a_2 = 0$.
Let also $\Xi_{\delta} = \Psi - \Phi_{\delta}$ where $\Psi$ is the solution of \eqref{eq_lin}. Then it satisfies
\begin{equation}\label{eq_U_delta}
\begin{cases}
- \Delta \Xi_{\delta} = \dfrac{9\delta\ep\pi_{ij}x_ix_j}{(|\bx|^2 + (x_4+1)^2)^{5 \over 2}} &\text{in } \R^4_+,\\
-\lim\limits_{x_4 \to 0} \dfrac{\pa \Xi_{\delta}}{\pa x_4} = 4 w_{1,0} \Xi_{\delta} + q_{\delta} &\text{on } \R^3
\end{cases}
\end{equation}
where
\[q_{\delta}(\bx) = \frac{\ep \pi_{ij}x_ix_j}{(|\bx|^2+1)^2} + \frac{ \delta \ep \pi_{ij}x_ix_j}{(|\bx|^2+1)^{5 \over 2}} \quad \text{on } \R^3.\]
\begin{lemma}
It holds that
\begin{multline}\label{eq_nFW_2}
F_m(W_{1,0}, \Psi_m) + F_m(\Psi_m, W_{1,0}) \\
\ge \(\frac{\pi}{24} + \frac{64}{105} \delta + O(\delta^2)\) \left|\S^2\right| \ep_m^2 \log (\rho\ep_m^{-1}) \|\pi_m\|^2 + O(\ep_m^2)
\end{multline}
for $\delta$ small.
\end{lemma}
\begin{proof}
Let $\Phi_{m,\delta}$ be the function $\Phi_{\delta}$ in \eqref{eq_Phi_delta} with $\ep = \ep_m$ and $\pi = \pi_m$.
Set $\Xi_{m,\delta}$ and $q_{m,\delta}$ in an analogous manner.
By \eqref{eq_Phi_n1} and (4.2) of \cite{KMW}, it holds that
\begin{equation}\label{eq_decay}
|\Phi_{m,\delta}(x)| + |\Xi_{m,\delta}(x)| \le \frac{C\ep_m |\pi_m|_{\infty}}{1+|x|}
\quad \text{and} \quad
|\nabla \Phi_{m,\delta}(x)| + |\nabla \Xi_{m,\delta}(x)| \le \frac{C\ep_m |\pi_m|_{\infty}}{1+|x|^2}
\end{equation}
where $|\pi_m|_{\infty} = \max_{i,j=1,2,3} |(\pi_m)_{ij}|$.
Integrating by parts, and employing \eqref{eq_decay},
\[\int_{B^3(0,\rho\ep_m^{-1})} \frac{d\bx}{1+|\bx|^3+(\rho\ep_m^{-1})^3} \le \left|\S^2\right| \int_0^{\rho\ep_m^{-1}} \frac{r^2 dr}{r^3+(\rho\ep_m^{-1})^3}
= \left|\S^2\right| \int_0^1 \frac{dt}{t^3+1} = O(1)\]
and
\[\int_0^{\rho\ep_m^{-1}} \int_{\pa B^3(0,\rho\ep_m^{-1})} \frac{dx}{1+|x|^3}
\le C \int_0^{\rho\ep_m^{-1}} \frac{(\rho\ep_m^{-1})^2 dx_4}{1+(\rho\ep_m^{-1})^3+x_4^3}
\le C\int_0^1 \frac{dt}{t^3+1} = O(1),\]
we calculate that
\begin{align}
&\ F_m(W_{1,0}, \Psi_m) + F_m(\Psi_m, W_{1,0}) \nonumber \\
&= - 2 \ep_m (\pi_m)_{ij} \int_0^{\rho\ep_m^{-1}} \int_{B^3(0,\rho\ep_m^{-1})} x_4 \left[\pa_{ij} W_{1,0} \(x_k \pa_k \Psi_m + x_4 \pa_4 \Psi_m + \Psi_m\) \right. \nonumber \\
&\hspace{130pt} \left. + \pa_{ij} \Psi_m \(x_k \pa_k W_{1,0} + x_4 \pa_4 W_{1,0} + W_{1,0}\) \right] d\bx dx_4 + O(\ep_m^2) \nonumber \\
&= 2 \ep_m (\pi_m)_{ij} \int_0^{\rho\ep_m^{-1}} \int_{B^3(0,\rho\ep_m^{-1})} x_4 \left[ 4\, \pa_i W_{1,0} \pa_j \Psi_m
+ \pa_i W_{1,0} (x_k \pa_{jk} \Psi_m + x_4 \pa_{j4} \Psi_m) \right. \label{eq_mt_32} \\
&\hspace{183pt} \left. + (x_k \pa_{ik} W_{1,0} + x_4 \pa_{i4} W_{1,0}) \pa_j \Psi_m \right] dx + O(\ep_m^2) \nonumber \\
&= - 2 \ep_m (\pi_m)_{ij} \int_0^{\rho\ep_m^{-1}} \int_{B^3(0,\rho\ep_m^{-1})} x_4 \pa_i W_{1,0} \pa_j \Psi_m dx + O(\ep_m^2) \nonumber \\
&= 2\ep_m (\pi_m)_{ij} \int_0^{\rho\ep_m^{-1}} \int_{B^3(0,\rho\ep_m^{-1})} \( x_4 \pa_{ij} W_{1,0} \Phi_{m,\delta} + x_4 \pa_{ij} W_{1,0} \Xi_{m,\delta} \) dx + O(\ep_m^2). \nonumber
\end{align}

On the other hand, by testing $\Xi_{m,\delta}$ in \eqref{eq_lin} and applying \eqref{eq_decay} once more, we obtain
\begin{multline*}
2\ep_m (\pi_m)_{ij} \int_0^{\rho\ep_m^{-1}} \int_{B^3(0,\rho\ep_m^{-1})} x_4 \pa_{ij} W_{1,0} \Xi_{m,\delta} dx \\
= \int_0^{\rho\ep_m^{-1}} \int_{B^3(0,\rho\ep_m^{-1})} \nabla \Phi_{m,\delta} \cdot \nabla \Xi_{m,\delta} dx
+ \int_0^{\rho\ep_m^{-1}} \int_{B^3(0,\rho\ep_m^{-1})} |\nabla \Xi_{m,\delta}|^2 dx + O(\ep_m^2).
\end{multline*}
Also, testing $\Phi_{m,\delta}$ in \eqref{eq_U_delta} shows
\begin{multline*}
\int_0^{\rho\ep_m^{-1}} \int_{B^3(0,\rho\ep_m^{-1})} \nabla \Xi_{m,\delta} \cdot \nabla \Phi_{m,\delta} dx \\
= \int_0^{\rho\ep_m^{-1}} \int_{B^3(0,\rho\ep_m^{-1})} \dfrac{9\delta\ep_m\pi_{ij}x_ix_j}{(|\bx|^2 + (x_4+1)^2)^{5 \over 2}} \Phi_{m,\delta} dx
+ \int_{B^3(0,\rho\ep_m^{-1})} q_{m,\delta} \Phi_{m,\delta} d\bx + O(\ep_m^2).
\end{multline*}
Consequently,
\begin{multline}\label{eq_mt_33}
2\ep_m (\pi_m)_{ij} \int_0^{\rho\ep_m^{-1}} \int_{B^3(0,\rho\ep_m^{-1})} x_4 \pa_{ij} W_{1,0} \Xi_m dx \\
\ge \int_0^{\rho\ep_m^{-1}} \int_{B^3(0,\rho\ep_m^{-1})} \dfrac{9\delta\ep_m\pi_{ij}x_ix_j}{(|\bx|^2 + (x_4+1)^2)^{5 \over 2}} \Phi_{m,\delta} dx
+ \int_{B^3(0,\rho\ep_m^{-1})} q_{m,\delta} \Phi_{m,\delta} d\bx + O(\ep_m^2).
\end{multline}

Combining \eqref{eq_mt_32} and \eqref{eq_mt_33}, we obtain
\begin{align}
F_m(W_{1,0}, \Psi_m) + F_m(\Psi_m, W_{1,0})
&\ge (\pi_m)_{ij} \( 2\ep_m \int_0^{\rho\ep_m^{-1}} \int_{B^3(0,\rho\ep_m^{-1})} x_4 \pa_{ij} W_{1,0} \Phi_{m,\delta} dx \right. \nonumber \\
&\hspace{50pt} + \int_0^{\rho\ep_m^{-1}} \int_{B^3(0,\rho\ep_m^{-1})} \dfrac{9\delta\ep_m\pi_{ij}x_ix_j}{(|\bx|^2 + (x_4+1)^2)^{5 \over 2}} \Phi_{m,\delta} dx \nonumber \\
&\hspace{50pt} \left. + \int_{B^3(0,\rho\ep_m^{-1})} q_{m,\delta} \Phi_{m,\delta} d\bx\) + O(\ep_m^2). \label{eq_nFW_3}
\end{align}

By applying \eqref{eq_Phi_delta}, we evaluate
\begin{equation}\label{eq_nFW_31}
\begin{aligned}
&\ 2\ep_m (\pi_m)_{ij} \int_0^{\rho\ep_m^{-1}} \int_{B^3(0,\rho\ep_m^{-1})} x_4 \pa_{ij} W_{1,0} \Phi_{m,\delta} dx \\
&= 16 \ep_m^2 (\pi_m)_{ij} (\pi_m)_{kl} \int_0^{\rho\ep_m^{-1}} \int_{B^3(0,\rho\ep_m^{-1})}
x_ix_jx_kx_l \\
&\hspace{100pt} \left[ \frac{x_4^2}{(|\bx|^2+(x_4+1)^2)^5} + \delta \frac{x_4}{(|\bx|^2 + (x_4+1)^2)^{9 \over 2}} \right] dx + O(\ep_m^2) \\
&= \frac{32}{15} \left|\S^2\right| \left[\int_0^{\rho\ep_m^{-1}} \frac{x_4^2 dx_4}{(x_4+1)^3} \int_0^{\infty} \frac{t^6 dt}{(t^2+1)^5} \right. \\
&\hspace{100pt} \left. + \delta \int_0^{\rho\ep_m^{-1}} \frac{x_4 dx_4}{(x_4+1)^2} \int_0^{\infty} \frac{t^6 dt}{(t^2+1)^{9 \over 2}} \right] \ep_m^2 \|\pi_m\|^2 + O(\ep_m^2) \\
&= \(\frac{\pi}{24} + \frac{32}{105} \delta\) \left|\S^2\right| \ep_m^2 \log(\rho\ep_m^{-1}) \|\pi_m\|^2 + O(\ep_m^2),
\end{aligned}
\end{equation}
\begin{equation}\label{eq_nFW_311}
\begin{aligned}
&\ 9\delta\ep_m (\pi_m)_{ij} \int_0^{\rho\ep_m^{-1}} \int_{B^3(0,\rho\ep_m^{-1})} \frac{x_ix_j}{(|\bx|^2 + (x_4+1)^2)^{5 \over 2}} \Phi_{m,\delta} dx \\
&= 9\delta\ep_m^2 (\pi_m)_{ij}(\pi_m)_{kl} \int_0^{\rho\ep_m^{-1}} \int_{B^3(0,\rho\ep_m^{-1})}
\dfrac{x_ix_jx_kx_lx_4}{(|\bx|^2 + (x_4+1)^2)^{9 \over 2}} dx \\
&\hspace{200pt} + O(\delta^2 \log(\rho\ep_m^{-1}) \|\pi_m\|^2) + O(\ep_m^2) \\
&= \frac{6}{35} \delta \left|\S^2\right| \ep_m^2 \log(\rho\ep_m^{-1}) \|\pi_m\|^2 + O(\delta^2 \log(\rho\ep_m^{-1}) \|\pi_m\|^2) + O(\ep_m^2) \end{aligned}
\end{equation}
and
\begin{equation}\label{eq_nFW_32}
\begin{aligned}
\int_{B^3(0,\rho\ep_m^{-1})} q_{m,\delta} \Phi_{m,\delta} d\bx
&= \ep_m^2 \delta (\pi_m)_{ij} (\pi_m)_{kl} \int_{B^3(0,\rho\ep_m^{-1})} \frac{x_ix_jx_kx_l}{(|\bx|^2 + 1)^{7 \over 2}} d\bx + O(\ep_m^2) \\
&= \frac{2}{15} \delta \left|\S^2\right| \ep_m^2 \log(\rho\ep_m^{-1}) \|\pi_m\|^2.
\end{aligned}
\end{equation}
Putting \eqref{eq_nFW_3}-\eqref{eq_nFW_32}, we deduce \eqref{eq_nFW_2}.
\end{proof}

\begin{cor}\label{cor_FW_3}
It holds that
\begin{multline*}
F_m(W_{1,0}, W_{1,0}) + \left[ F_m(W_{1,0}, \Psi_m) + F_m(\Psi_m, W_{1,0}) \right] \\
\ge \(\frac{64}{105} \delta + O(\delta^2)\) \left|\S^2\right| \ep_m^2 \log (\rho\ep_m^{-1}) \|\pi_m\|^2 + O(\ep_m^2)
\end{multline*}
for $\delta$ small.
\end{cor}
\begin{proof}
The result immediately follows from \eqref{eq_nFW} and \eqref{eq_nFW_2}.
\end{proof}

\begin{proof}[Completion of the proof of Proposition \ref{prop_van} for $N = 4$]
By taking $\delta > 0$ in Corollary \ref{cor_FW_3} small enough, we infer from \eqref{eq_poho_3}, \eqref{eq_poho_31} and \eqref{eq_poho_32} that
\[O(\ep_m^2) \ge \frac{32}{105} \delta \left|\S^2\right| \ep_m^2 \log(\rho\ep_m^{-1}) \|\pi_m\|^2 + O(\ep_m^2).\]
Accordingly,
\[O\(\frac{1}{|\log\ep_m|}\) \ge \|\pi_m\|^2 + O\(\frac{1}{|\log\ep_m|}\).\]
This implies that \eqref{eq_van} holds for $N = 4$.
\end{proof}

\subsection{Non-negativity of a sum of the second-order derivatives of the trace-free second fundamental form} \label{subsec_nonneg}
To derive Proposition \ref{prop_van}, we analyzed the $\ep_m^2$-order of the asymptotic expansion of the term $\mcp_{1m}\(\wtv_m, \rho\ep_m^{-1}\)$.
We will prove the next result by examining its $\ep_m^3 |\log \ep_m|$-order.
\begin{prop}\label{prop_nonneg}
Suppose that $N = 5$ and $y_m \to y_0 \in \pa M$ is an isolated simple blow-up point of the sequence $\{U_m\}_{m \in \N}$ of the solutions to \eqref{eq_Yamabe_3}.
If $\{\tig_m\}_{m \in \N}$ is a sequence of the metrics constructed in Subsection \ref{subsec_blow}, then
\begin{equation}\label{eq_nonneg}
\pi[\tig_0]_{ij,ij}(y_0) \ge 0.
\end{equation}
\end{prop}
\begin{proof}
Fix any $\rho \in (0, \rho_5]$. By appealing \eqref{eq_van}, one can improve the error in \eqref{eq_mt} so that
\[\mcp_{1m}\(\wtv_m, \rho\ep_m^{-1}\) \\
= F_m(W_{1,0}, W_{1,0}) + \left[ F_m(W_{1,0}, \Psi_m) + F_m(\Psi_m, W_{1,0}) \right] + O(\ep_m^3)\]
where $F_m$ is the map defined in \eqref{eq_Fm}.
From this, \eqref{eq_poho_3}, \eqref{eq_poho_31} and \eqref{eq_poho_32}, we deduce
\begin{equation}\label{eq_nonneg_1}
O(\ep_m^3) \ge F_m(W_{1,0}, W_{1,0}) + \left[ F_m(W_{1,0}, \Psi_m) + F_m(\Psi_m, W_{1,0}) \right] + O(\ep_m^3).
\end{equation}
Moreover, arguing as in the proof of Lemma \ref{eq_nW}, we see
\begin{align*}
&\ F_m(W_{1,0},W_{1,0}) + \frac{1}{64} \left|\S^3\right| \ep_m^2 \|\pi_m\|^2 \\
&= - \ep_m^3 (\pi_m)_{ij,kl} \int_{B^5_+(0,\rho\ep_m^{-1})} x_5x_kx_l \pa_{ij} W_{1,0} Z_{1,0}^0 dx \\
&\ - 2 \ep_m^3 (\pi_m)_{ij,ik} \int_{B^5_+(0,\rho\ep_m^{-1})} x_5x_k \pa_j W_{1,0} Z_{1,0}^0 dx + O(\ep_m^3) \\
&= - \frac{15}{8} \ep_m^3 (\pi_m)_{ij,ij} \int_0^{\rho\ep_m^{-1}} \int_{\R^4} \frac{x_5|\bx|^4 (1-|\bx|^2-x_5^2)}{(|\bx|^2+(x_5+1)^2)^6} dx \\
&\ + \frac{9}{4} \ep_m^3 (\pi_m)_{ij,ij} \int_0^{\rho\ep_m^{-1}} \int_{\R^4} \frac{x_5 |\bx|^2 (1-|\bx|^2-x_5^2)}{(|\bx|^2+(x_5+1)^2)^5} dx + O(\ep_m^3) \\
&= -\frac{9}{64} \left|\S^3\right| \ep_m^3 \log(\rho \ep_m^{-1}) (\pi_m)_{ij,ij} + O(\ep_m^3).
\end{align*}
To deduce each equality, we took $\kappa \ge 4$ in Lemma \ref{lemma_conf} so that $R_{NN,N}[\tig_m](y_m) = 0$, and used Lemma \ref{lemma_metric}, the symmetry of the integral and \eqref{eq_van}.
After setting $a_1 = -\frac{63}{4}$ and $a_2 = \frac{105}{8}$ as in the proof of Corollary \ref{cor_FW} and applying \eqref{eq_van} once more, we arrive at
\begin{align*}
&\ F_m(W_{1,0}, W_{1,0}) + \left[ F_m(W_{1,0}, \Psi_m) + F_m(\Psi_m, W_{1,0}) \right] \\
&\ge \frac{3}{2560} \left|\S^3\right| \ep_m^2 \|\pi_m\|^2 -\frac{9}{64} \left|\S^3\right| \ep_m^3 \log(\rho \ep_m^{-1}) (\pi_m)_{ij,ij} + O(\ep_m^3) + O(\ep_m^3 |\log \ep_m| \cdot \|\pi_m\|) \\
&\ge -\frac{9}{64} \left|\S^3\right| \ep_m^3 \log(\rho \ep_m^{-1}) (\pi_m)_{ij,ij} + O(\ep_m^3) + O\(\ep_m^{7 \over 2} |\log \ep_m|^{3 \over 2}\).
\end{align*}
Inserting this estimate to \eqref{eq_nonneg_1}, we obtain
\[O\(\frac{1}{|\log\ep_m|}\) \ge -\frac{9}{64} \left|\S^3\right| (\pi_m)_{ij,ij} + O\(\frac{1}{|\log\ep_m|}\) + O\(\ep_m^{1 \over 2} |\log \ep_m|^{3 \over 2}\).\]
Taking $m \to \infty$ on the both sides, we obtain \eqref{eq_nonneg}.
\end{proof}

\section{Local sign restriction and set of blow-up points}\label{sec_lsr}
Under the validity of Proposition \ref{prop_van}, we derive the local sign restriction of the function $\mcp'$.
\begin{prop}\label{prop_lsr}
Assume that $N \ge 4$ and $y_m \to y_0 \in \pa M$ is an isolated simple blow-up point for the sequence $\{U_m\}_{m \in \N}$ to the solutions to \eqref{eq_Yamabe_3}.
Then, given $m \in \N$ large and $\rho > 0$ small, there exist constants $\mcc_0 \ge 0$ and $\mcc_1,\, \mcc_2,\, \mcc_3 > 0$ independent of $m$ and $\rho$ such that
\[\ep_m^{N-2+o(1)} \mcp'\(\wtu_m(0)\wtu_m, \rho\) \ge \ep_m^2 \mcc_0 - \ep_m^{2+\eta} \rho^{2-\eta} \mcc_1 - \ep_m^{N-2} \rho^{-N+3} \mcc_2
- {\ep_m^{N-1} \rho^{N-1} \mcc_3 \over \ep_m^{2(N-1)+o(1)} + \rho^{2(N-1)+o(1)}}\]
for $N \ge 5$ and
\[\ep_m^{2+o(1)} \mcp'\(\wtu_m(0)\wtu_m, \rho\) \ge \ep_m^2 \log(1+\rho\ep_m^{-1})\, \mcc_0 - \ep_m^2 \mcc_1
- {\ep_m^3 \rho^3 \mcc_2 \over \ep_m^{6+o(1)} + \rho^{6+o(1)}}\]
for $N = 4$, in $\tig_m$-Fermi coordinates centered in $y_m$.
Here, $\mcp'$ is the function defined in \eqref{eq_poho_0}, $\eta > 0$ is an arbitrarily small number and $\ep_m^{o(1)} \to 1$ as $m \to \infty$.
\end{prop}
\begin{proof}
If $N \ge 5$, the proof follows the same lines as that of Lemma 6.1 in \cite{KMW}; cf. Theorem 7.2 of \cite{Al}.
Slightly modifying the argument, one can also establish the inequality for $N = 4$.
Here we allow the possibility that $\pi[\tig_0](y_0) = 0$ as opposed to \cite{KMW}. Thus we cannot exclude that $\mcc_0 = 0$.
\end{proof}

From the previous proposition, we conclude the following results. It can be derived as in Section 6 of \cite{KMW}.
\begin{lemma}
Assume that $N \ge 4$, and $y_0 \in \pa M$ is an isolated blow-up point for the sequence $\{U_m\}_{m \in \N}$ to \eqref{eq_Yamabe_3}.
Then it is an isolated simple blow-up point of $\{U_m\}_{m \in \N}$. \end{lemma}
\begin{prop}\label{prop_sim}
Assume the hypotheses of Theorem \ref{thm_main}.
Let $\vep_0, \vep_1,\, R,\, C_0$ and $C_1$ be positive numbers in the statement of Proposition \ref{prop_blow_a}.
Suppose that $U \in H^1(M)$ is a solution to \eqref{eq_Yamabe_3} and $\{y_1, \cdots, y_{\mcn}\}$ is the set of its local maxima on $\pa M$.
Then there exists a constant $C_2 > 0$ depending only on $(M,g)$, $N$, $\vep_0$, $\vep_1$ and $R$ such that if $\max_{\pa M} U \ge C_0$,
then $d_h(y_{m_1}, y_{m_2}) \ge C_2$ for all $1 \le m_1 \ne m_2 \le \mcn(U)$.
In particular, the set of blow-up points of $\{U_m\}_{m \in \N}$ is finite and it consists of isolated simple blow-up points.
\end{prop}

\section{The compactness result}\label{sec_cpt}
Let $G_{y_0}$ be the normalized Green's function of the conformal Laplacian on $(M, \tig_0)$ with Neumann boundary condition with pole at $y_0 \in \pa M$, that is, the solution of
\begin{equation}\label{eq_Green}
\begin{cases}
L_{\tig_0} G_{y_0}(y) = 0 &\text{in } M, \\
B_{\tig_0} G_{y_0}(y) = \delta_{y_0} &\text{on } \pa M, \\
\lim\limits_{d_{\tig_0}(y,y_0) \to 0} d_{\tig_0}(y,y_0)^{N-2} G_{y_0}(y) = 1
\end{cases}
\end{equation}
where $\delta_{y_0}$ is the Dirac measure centered at $y_0$.
It will serve as the function $G$ when we apply the positive mass theorem (described in Lemma \ref{lemma_pmt}).

In the following two lemmas, we verify the necessary conditions to apply Lemma \ref{lemma_pmt} for 4- and 5-manifolds.
Note that the number $d$ in \eqref{eq_g_ab} is 1.
\begin{lemma}\label{lemma_exp}
Suppose that $N = 4$ or $5$, $y_0 \in \pa M$ is an isolated simple blow-up point of the sequence $\{U_m\}_{m \in \N}$ of the solutions to \eqref{eq_Yamabe_3}.
If we take $\kappa \ge 4$ in \eqref{eq_kappa}, we can expand the metric $g = \tig_0$ as in \eqref{eq_g_ab} and \eqref{eq_A_ab}.
\end{lemma}
\begin{proof}
By Lemma \ref{lemma_metric} and Proposition \ref{prop_van}, it clearly holds that
\[A_{iN}(x) = A_{NN}(x) = 0 \quad \text{and} \quad A_{ij}(x) = O(|x|^2).\]
Therefore,
\[\exp A(x) = I + A(x) + O(|x|^4) \quad \text{and so} \quad g(x) = \exp A(x) + O(|x|^4)\]
where $I$ is the $N \times N$-identity matrix. From this, we see that
\[\det g(x) = e^{\text{trace} A(x) + O(|x|^4)} = 1 + \text{trace} (A(x)) + O(|x|^4).\]
By virtue of our choice $\kappa \ge 4$, it follows that $\text{trace}(A(x)) = O(|x|^4)$ as desired.
\end{proof}

\begin{lemma}\label{lemma_Green}
Suppose that $N = 4$ or $5$, and $y_0 \in \pa M$ is an isolated simple blow-up point of the sequence $\{U_m\}_{m \in \N}$ of the solutions to \eqref{eq_Yamabe_3}.
If we choose the integer $\kappa$ in \eqref{eq_kappa} large enough, we obtain that
\[G_{y_0}(x) = \begin{cases}
|x|^{-2} + O(|\log|x||) &\text{for } N = 4,\\
|x|^{-3} + O(|x|^{-1}|\log|x||) &\text{for } N = 5
\end{cases}\]
in $\tig_0$-Fermi coordinates centered at $y_0$.
As a particular consequence, $G_{y_0}$ is a smooth positive function on $M \setminus \{y_0\}$ which can be expressed as in \eqref{eq_Green_ex}-\eqref{eq_phi}.
\end{lemma}
\begin{proof}
We will employ Proposition B.2 of \cite{AS}, in which Almaraz and Sun constructed the Green's function on manifolds with boundary using parametrices.

According to their result, if there exists a sufficiently large integer $\kappa_0$ such that
\begin{equation}\label{eq_H}
|H[\tig_0](y)| \le Cd_{\tig_0}(y,y_0)^{\kappa_0} \quad \text{for all } y \in \pa M,
\end{equation}
then one can find a smooth positive solution $G_{y_0}$ on $M \setminus \{y_0\}$ to \eqref{eq_Green} with $g = \tig_0$.
Moreover, if $\tig_0 = \exp B$ for some $2$-tensor $B$ on $M$, then
\begin{multline}\label{eq_G_rem}
\left|G_{y_0}(x) - |x|^{2-N}\right| \\
\le C \sum_{i,j=1}^n \sum_{|\alpha|=1}^d \left|B_{ij,\alpha}(0)\right| |x|^{|\alpha|+2-N} + \begin{cases}
C (1+|\log|x||) &\text{for } N = 3, 4,\\
C |x|^{d+3-N} &\text{for } N \ge 5
\end{cases}
\end{multline}
in $\tig_0$-Fermi coordinates centered at $y_0$, where $d = \lfloor \frac{N-2}{2} \rfloor$ as before.
Check also \eqref{eq_mi} for the notations involving multi-indices.

Differentiating (3.4) of \cite{Es} $|\beta|$-times, we obtain
\begin{equation}\label{eq_detg_d}
\frac{\pa}{\pa x_N} \frac{\pa \sqrt{|\tig_0|}}{\pa x_{\beta}} (\bx,0)
= - n \sum_{\beta'+\beta''=\beta} \frac{\beta!}{\beta'!\beta''!} \(\frac{\pa \sqrt{|\tilde{h}_0|}}{\pa x_{\beta'}} \frac{\pa H[\tig_0]}{\pa x_{\beta''}} \)(\bx)
\quad \text{for } \bx \in \R^n,
\end{equation}
in normal coordinates on $\pa M$ centered at $y_0$. Here $\tilde{h}_0$ is the restriction of $\tig_0$ to $\pa M$.
In light of \eqref{eq_kappa} and \eqref{eq_detg_d}, the coefficient of $x_{\beta} x_N$ in the Taylor expansion of $\sqrt{|\tig_0|}$ at $x = 0$ has to be
\[-\frac{n}{(|\beta|+1)!} \frac{\pa^{\beta}H[\tig_0]}{\pa x_{\beta}} (0) = 0 \quad \text{for all } |\beta| \le \kappa-1.\]
Thus, if we take $\kappa \ge \kappa_0$, all partial derivatives of $H$ of order $\le \kappa_0-1$ must vanish at $0$, and so \eqref{eq_H} holds.

On the other hand, we know that $A(x) = O(|x|^2)$ and $\tig_0 = \exp A + O(|x|^4)$ from the proof of Lemma \ref{lemma_exp}. Therefore, for $|\alpha| = 1$,
\[B_{ij,\alpha}(x) = A_{ij,\alpha}(x) + O(|x|) = O(|x|), \quad \text{and so } B_{ij,\alpha}(0) = 0.\]
This implies that the right-hand side of \eqref{eq_G_rem} is bounded by
\[\begin{cases}
C (|x|^{4-N} + 1 + |\log|x||) = C (1+|\log|x||) = O(|\log|x||) &\text{for } N = 4,\\
C |x|^{4-N} = C |x|^{-1} = O(|x|^{-1} |\log|x||) &\text{for } N = 5.
\end{cases}\]
The proof is finished.
\end{proof}

We next examine the relationship between the flux integral $\mci(y_0,\rho)$ given in \eqref{eq_mci} and the quantity $\mcp'(G_{y_0},\rho)$ defined by \eqref{eq_poho_0}.
\begin{lemma}\label{lemma_exp_2}
Under the assumptions of Lemma \ref{lemma_exp}, it holds that
\begin{equation}\label{eq_PI}
\mcp'(G_{y_0},\rho) = \begin{cases}
- \dfrac{1}{6} \mci(y_0,\rho) + O(\rho) &\text{for } N = 4,\\
- \dfrac{9}{32} \mci(y_0,\rho) - \dfrac{3}{512} \left|\S^3\right| \pi[\tig_0]_{ij,ij}(y_0) + O(\rho|\log\rho|) &\text{for } N = 5.
\end{cases}
\end{equation}
\end{lemma}
\begin{proof}
Lemma 3.2 of \cite{AdQW} leads us that
\begin{equation}\label{eq_mcp_5}
\begin{aligned}
\mcp'(G_{y_0},\rho) &= -\(\frac{N-2}{2}\) \int_{\pa_I B^N_+(0,\rho)} \(|x|^{2-N} \pa_a G_{y_0}(x) - \pa_a|x|^{2-N} G_{y_0}(x)\) \frac{x_a}{|x|}\, dS_x \\
&\ + O(\rho^{6-N}|\log\rho|).
\end{aligned}
\end{equation}
Therefore, we infer from \eqref{eq_mci} that
\begin{multline}\label{eq_mcp_3}
\mcp'(G_{y_0},\rho) = -\frac{(N-2)^2}{8(N-1)} \left[ \mci(y_0,\rho) \right. \\
\left. + \int_{\pa_I B^N_+(0,\rho)} \(\rho^{3-2N} x_a\pa_b A_{ab}(x) - 2N \rho^{1-2N} x_ax_bA_{ab}(x)\) dS_x \right] + O(\rho^{6-N}|\log\rho|).
\end{multline}

On the other hand, by setting $\kappa \ge 4$ in Lemma \ref{lemma_conf} and applying Proposition \ref{prop_van}, we obtain
\begin{equation}\label{eq_conf_2}
\text{Sym}_{klm}R_{ikjl,m}[\tih_0] = \text{Sym}_{kl} H_{,kl}[\tig_0] = R_{NN,k}[\tig_0] = R_{NN,N}[\tig_0] = 0 \quad \text{at } y_0
\end{equation}
where $\tih_0 = \tig_0|_{T\pa M}$. Thanks to \eqref{eq_A_ab}, \eqref{eq_conf}, \eqref{eq_conf_2}, the Ricci identity
and the symmetry of the integral, we find
\begin{multline}\label{eq_mcp_4}
\int_{\pa_I B^N_+(0,\rho)} \(\rho^{3-2N} x_i\pa_j A_{ij}(x) - 2N \rho^{1-2N} x_ix_jA_{ij}(x)\) dS_x \\
= \rho^{5-N}|\S^{n-1}| \left[\frac{2(N-3)}{(N-1)(N+1)(N+3)}\right] \pi[\tig_0]_{ij,ij}(y_0) + O(\rho^{6-N})
\end{multline}
for any $N \ge 4$.

Combining \eqref{eq_mcp_3} and \eqref{eq_mcp_4}, we derive \eqref{eq_PI}.
\end{proof}
\noindent The above lemma shows that Proposition 3.6 of \cite{AdQW} is valid for $N = 4$, but is not in general for $N = 5$.

As a by-product of the previous lemma, we can evaluate the mass $m_0$.
\begin{cor}\label{cor_mass}
Under the assumptions of Lemma \ref{lemma_exp}, it holds that
\[m_0 = \begin{cases}
-6 \lim\limits_{\rho \to 0} \mcp'(G_{y_0},\rho) &\text{for } N = 4,\\
-\dfrac{32}{9} \lim\limits_{\rho \to 0} \mcp'(G_{y_0},\rho) - \dfrac{1}{48} \left|\S^3\right| \pi[\tig_0]_{ij,ij}(y_0) &\text{for } N = 5.
\end{cases}\]
\end{cor}
\begin{proof}
Taking $\rho \to 0$ on the both sides of \eqref{eq_PI} and using \eqref{eq_pmt}, we get the result.
\end{proof}
\noindent One can see from \eqref{eq_poho_0} or \eqref{eq_mcp_5} that the value of $\mcp'(G_{y_0},\rho)$ is completely determined by the Green's function $G_{y_0}$.
Therefore, the above corollary tells us that the mass is involved with not only the Green's function but also the trace-free second fundamental form if $N = 5$.
As mentioned in Remark \ref{rmk_mar}, it is a unique property of manifolds with boundary.

\medskip
We are now ready to complete the proof of our main result.
\begin{proof}[Proof of Theorem \ref{thm_main}]
Suppose that $y_0 \in \pa M$ is a blow-up point of of the sequence $\{U_m\}_{m \in \N}$ of the solutions to \eqref{eq_Yamabe_3}.
By Proposition \ref{prop_sim}, it is isolated simple.
By Proposition \ref{prop_iso} and elliptic regularity theory, there also exists a constant $a > 0$ such that
\[U_m(y_m)U_m \to aG_{y_0} \quad \text{in } C^2(B_{\tig_0}(y_0,\rho) \setminus \{y_0\}) \quad \text{as } m \to \infty.\]
Thanks to Proposition \ref{prop_lsr}, it follows that
\begin{equation}\label{eq_mcp_2}
\liminf_{\rho \to 0} \mcp'(aG_{y_0},\rho) = a^2 \liminf_{\rho \to 0} \mcp'(G_{y_0},\rho) \ge 0.
\end{equation}
We split the proof into two cases according to the dimension of the manifold $M$.

\medskip \noindent
\textsc{Case $N = 4$ and $5$:} By virtue of Lemmas \ref{lemma_exp} and \ref{lemma_Green}, all the conditions needed to apply Lemma \ref{lemma_pmt} hold.
Besides, \eqref{eq_Green} yields that $R\big[G_{y_0}^{4 \over N-2}g\big] = 0$ and $H\big[G_{y_0}^{4 \over N-2}g\big] = 0$ on their respective domains, which trivially implies \eqref{eq_pmt_c}.
Employing Lemma \ref{lemma_pmt}, Corollary \ref{cor_mass}, Proposition \ref{prop_nonneg} and \eqref{eq_mcp_2}, we deduce
\[0 < m_0 = \begin{cases}
-6 \lim\limits_{\rho \to 0} \mcp'(G_{y_0},\rho) \le 0 &\text{for } N = 4,\\
-\dfrac{32}{9} \lim\limits_{\rho \to 0} \mcp'(G_{y_0},\rho) - \dfrac{1}{48} \left|\S^3\right| \pi[\tig_0]_{ij,ij}(y_0) \le 0 &\text{for } N = 5,
\end{cases}\]
a contradiction. Consequently, there is no blow-up point of a solution to \eqref{eq_Yamabe_3},
which means that its solution set is $L^{\infty}(M)$-bounded.
Elliptic regularity tells us that it is $C^2(M)$-compact. Theorem \ref{thm_main} must be true in this case.

\medskip \noindent
\textsc{Case $N = 6$:} We remind that the trace-free second fundamental form $\pi[g]$ is assumed to be never zero on $\pa M$.
There is a positive smooth function $\omega_0$ on $M$ such that $\tig_0 = \omega_0 g$ on $M$.
In Proposition 1.2 of \cite{Es4}, it was proved that $\pi[\tig_0] = \sqrt{\omega}\, \pi[g]$ on $\pa M$.
This produce a contradiction, since Proposition \ref{prop_van} reads
\[0 = \|\pi[\tig_0](y_0)\| = \omega(y_0)^{-{1 \over 2}} \|\pi[g](y_0)\| > 0.\]
The same reasoning as above shows that Theorem \ref{thm_main} is also valid in this case.
\end{proof}

\appendix
\section{Proof of Lemmas \ref{lemma_Phi} and \ref{lemma_U}}\label{sec_app_PU}
Throughout this section, we assume that $N \ge 5$. The case $N = 4$ can be handled similarly.

\medskip
In order to prove Lemma \ref{lemma_Phi}, we first need two preliminary observations.
\begin{lemma}\label{lemma_Phi_1}
Suppose that $N \ge 5$. The function
\[\Phi_1(x) = \frac{1}{4(N-4)} \frac{x_N+1}{(|\bx|^2 + (x_N+1)^2)^{N-4 \over 2}} + a_1 \frac{x_N+1}{(|\bx|^2 + (x_N+1)^2)^{N \over 2}} \quad \text{in } \R^N_+\]
for $a_1 \in \R$ satisfies
\[-\Delta \Phi_1 = \frac{x_N+1}{(|\bx|^2 + (x_N+1)^2)^{N-2 \over 2}} \quad \text{in } \R^N_+.\]
\end{lemma}
\begin{proof}
It holds that
\[\frac{x_N+1}{(|\bx|^2 + (x_N+1)^2)^{N-2 \over 2}} = -\(\frac{1}{N-4}\) \pa_N \left[ \frac{1}{(|\bx|^2 + (x_N+1)^2)^{N-4 \over 2}} \right] \quad \text{in } \R^N_+.\]
Thus, if we have a solution $\Phi_0$ of the equation
\[-\Delta \Phi_0 = \frac{1}{(|\bx|^2 + (x_N+1)^2)^{N-4 \over 2}} \quad \text{in } \R^N_+,\]
we will be able to choose
\begin{equation}\label{eq_Phi_11}
\Phi_1 = -\(\frac{1}{N-4}\) \pa_N \Phi_0.
\end{equation}

On the other hand, we see that
\[-\Delta [\Phi_0(\bx, x_N-1)] = \frac{1}{(|\bx|^2 + x_N^2)^{N-4 \over 2}} = \frac{1}{|x|^{N-4}} \quad \text{in } \R^n \times (1, \infty).\]
If we assume that $\Phi_0(\bx, x_N-1)$ is radial symmetric, i.e., $\phi_0(|x|) = \Phi_0(\bx, x_N-1)$, then it is reduced to
\[- \phi_0'' - \frac{N-1}{r} \phi_0' = \frac{1}{r^{N-4}} \quad \text{in } (0, \infty).\]
Its general solution is expressed as
\[\phi_0(r) = \begin{cases}
\dfrac{1}{4(N-6)} \dfrac{1}{r^{N-6}} + \dfrac{a_1}{r^{N-2}} + a_1' &\text{for } N = 5 \text{ or } N \ge 7,\\
-\dfrac{\log r}{4} + \dfrac{a_1}{r^4} + a_1' &\text{for } N = 6
\end{cases}\]
for $r \in (0,\infty)$ and $a_1,\, a_1' \in \R$. Consequently,
\begin{equation}\label{eq_Phi_0}
\Phi_0(x) = \begin{cases}
\dfrac{1}{4(N-6)} \dfrac{1}{(|\bx|^2 + (x_N+1)^2)^{N-6 \over 2}} + \dfrac{a_1}{(|\bx|^2 + (x_N+1)^2)^{N-2 \over 2}} + a_2 \\
\hspace{250pt} \text{for } N = 5 \text{ or } N \ge 7,\\
-\dfrac{1}{8} \log (|\bx|^2 + (x_N+1)^2) + \dfrac{a_1}{(|\bx|^2 + (x_N+1)^2)^2} + a_1' \\
\hspace{250pt} \text{for } N = 6
\end{cases}
\end{equation}
in $\R^N_+$.

By \eqref{eq_Phi_11} and \eqref{eq_Phi_0}, the assertion in the statement holds.
\end{proof}

\begin{lemma}\label{lemma_Phi_2}
The function
\[\Phi_2(x) = \frac{1}{2(N-4)} \frac{1}{(|\bx|^2 + (x_N+1)^2)^{N-4 \over 2}} + a_2 \frac{1}{(|\bx|^2 + (x_N+1)^2)^{N-2 \over 2}} + a_2' \quad \text{in } \R^N_+\]
for $a_2,\, a_2' \in \R$ satisfies
\begin{equation}\label{eq_Phi_2}
-\Delta \Phi_2 = \frac{1}{(|\bx|^2 + (x_N+1)^2)^{N-2 \over 2}} \quad \text{in } \R^N_+.
\end{equation}
\end{lemma}
\begin{proof}
Equation \eqref{eq_Phi_2} is equivalent to
\[-\Delta [\Phi_2(\bx,x_N-1)] = \frac{1}{(|\bx|^2 + x_N^2)^{N-2 \over 2}} = \frac{1}{|x|^{N-2}} \quad \text{in } \R^N_+.\]
If we assume that $\Phi_2(\bx, x_N-1)$ is radial symmetric, i.e., $\phi_2(|x|) = \Phi_2(\bx, x_N-1)$, then it is reduced to
\[- \phi_2'' - \frac{n}{r} \phi_2' = \frac{1}{r^{N-2}} \quad \text{in } (0, \infty).\]
The general solution is expressed as
\[\phi_2(r) = \frac{1}{2(N-4)r^{N-4}} + \frac{a_2}{r^{N-2}} + a_2'\]
for $r \in (0,\infty)$ and $a_2,\, a_2' \in \R$. As a result, the assertion in the statement holds.
\end{proof}

\begin{cor}\label{cor_Phi}
The function
\[\begin{aligned}
(\Phi_1 - \Phi_2)(x) &= \frac{1}{4(N-4)} \frac{x_N-1}{(|\bx|^2 + (x_N+1)^2)^{N-4 \over 2}} + a_1 \frac{x_N+1}{(|\bx|^2 + (x_N+1)^2)^{N \over 2}} \\
&+ a_2 \frac{1}{(|\bx|^2 + (x_N+1)^2)^{N-2 \over 2}} + a_2'
\end{aligned} \quad \text{in } \R^N_+\]
for $a_1,\, a_2,\, a_2' \in \R$ satisfies
\[-\Delta (\Phi_1 - \Phi_2) = \frac{x_N}{(|\bx|^2 + (x_N+1)^2)^{N-2 \over 2}} = x_N W_{1,0} \quad \text{in } \R^N_+.\]
\end{cor}
\begin{proof}
It is a direct consequence of Lemmas \ref{lemma_Phi_1} and \ref{lemma_Phi_2}.
\end{proof}

\begin{proof}[Completion of the proof of Lemma \ref{lemma_Phi}]
Define $\wPh_{ij} = \pa_{ij} (\Phi_1 - \Phi_2)$ so that $\Phi = 2\pi_{ij} \wPh_{ij}$. By Corollary \ref{cor_Phi},
\[\begin{aligned}
\wPh_{ij}(x) &= -\frac{x_N-1}{4} \left[ \frac{\delta_{ij}}{(|\bx|^2 + (x_N+1)^2)^{N-2 \over 2}} - (N-2)\frac{x_ix_j}{(|\bx|^2 + (x_N+1)^2)^{N \over 2}} \right] \\
&\ + a_1(x_N+1) \left[ \frac{\delta_{ij}}{(|\bx|^2 + (x_N+1)^2)^{N+2 \over 2}} - (N+2)\frac{x_ix_j}{(|\bx|^2 + (x_N+1)^2)^{N+4 \over 2}} \right] \\
&\ + a_2 \left[ \frac{\delta_{ij}}{(|\bx|^2 + (x_N+1)^2)^{N \over 2}} - N\frac{x_ix_j}{(|\bx|^2 + (x_N+1)^2)^{N+2 \over 2}} \right]
\end{aligned}
\quad \text{in } \R^N_+.\]
Since the trace of $\pi$ is assumed to be 0, we have \eqref{eq_Phi_n1}.
This completes the proof.
\end{proof}

\begin{proof}[Completion of the proof of Lemma \ref{lemma_U}]
It follows from \eqref{eq_lin} and \eqref{eq_Phi_n} that $U$ is harmonic in $\R^N_+$. Note also that
\[\lim_{x_N \to 0} \dfrac{\pa U}{\pa x_N} + N w_{1,0}^{2 \over N-2} U
= - \lim_{x_N \to 0} \dfrac{\pa \Phi}{\pa x_N} - N w_{1,0}^{2 \over N-2} \Phi \quad \text{on } \R^n.\]
Plugging \eqref{eq_Phi_n1} into the right-hand side, we find the boundary condition that $U$ satisfies.
\end{proof}

\medskip \noindent
\textbf{Acknowledgement.} S. Kim is supported by Basic Science Research Program through the National Research Foundation of Korea(NRF) funded by the Ministry of Education (NRF2017R1C1B5076384),
and the associate member problem of Korea institute for advanced study(KIAS).
M. Musso has been supported by Fondecyt grant 1160135.
The research of J. Wei is partially supported by NSERC of Canada.


\begin{thebibliography}{0}
\bibitem{Al2}
\textsc{S. Almaraz,}
An existence theorem of conformal scalar-flat metrics on manifolds with boundary, Pacific J. Math. \textbf{248} (2010), 1--22.

\bibitem{Al}
\bysame, A compactness theorem for scalar-flat metrics on manifolds with boundary, Calc. Var. Partial Differential Equations \textbf{41} (2011), 341--386.

\bibitem{Al3}
\bysame, Blow-up phenomena for scalar-flat metrics on manifolds with boundary, J. Differential Equations \textbf{251} (2011), 1813--1840.

\bibitem{ABd}
\textsc{S. Almaraz, E. Barbosa, L. L. de Lima,}
A positive mass theorem for asymptotically flat manifolds with a non-compact boundary, Comm. Anal. Geom. \textbf{24} (2016), 673--715.

\bibitem{AdQW}
\textsc{S. Almaraz, O. S. de Queiroz, S. Wang,}
A compactness theorem for scalar-flat metrics on 3-manifolds with boundary, J. Funct. Anal. in press.

\bibitem{AS}
\textsc{S. Almaraz, L. Sun,}
Convergence of the Yamabe flow on manifolds with minimal boundary, Ann. Sc. Norm. Super. Pisa Cl. Sci., in press.

\bibitem{Au}
\textsc{T. Aubin,}
\'{E}quations diff\'{e}rentielles non lin\'{e}aires et probl\`{e}me de Yamabe concernant la courbure scalaire, J. Math. Pures Appl. \textbf{55} (1976), 269--296.

\bibitem{Br}
\textsc{S. Brendle,}
Blow-up phenomena for the Yamabe equation, J. Amer. Math. Soc. \textbf{21} (2008), 951--979.

\bibitem{BC}
\textsc{S. Brendle, S. Chen,}
An existence theorem for the Yamabe problem on manifolds with boundary, J. Eur. Math. Soc. (JEMS) \textbf{16} (2014), 991--1016.

\bibitem{BM}
\textsc{S. Brendle, F. Marques,}
Blow-up phenomena for the Yamabe equation II, J. Differential Geom. \textbf{81} (2009), 225--250.

\bibitem{CS}
\textsc{E. C\'ardenas, W. Sierra,}
Uniqueness of solutions of the Yamabe problem on manifolds with
boundary, Nonlinear Anal. \textbf{187} (2019), 125--133.

\bibitem{Ch}
\textsc{S. S. Chen,}
Conformal deformation to scalar flat metrics with constant mean curvature on the boundary in higher dimensions, preprint, arXiv:0912.1302.

\bibitem{Cr}
\textsc{P. Cherrier,}
Probl\`emes de Neumann non lin\'eaires sur les vari\'et\'es Riemannienes, J. Funct. Anal. \textbf{57} (1984) 154--206.

\bibitem{DDS}
\textsc{J. D\'avila, M. del Pino, Y. Sire,}
Nondegeneracy of the bubble in the critical case for nonlocal equations, Proc. Amer. Math. Soc. \textbf{141} (2013), 3865--3870.

\bibitem{DKP}
\textsc{S. Deng, S. Kim, A. Pistoia,}
Linear perturbations of the fractional Yamabe problem on the minimal conformal infinity, to appear in Comm. Anal. Geom.

\bibitem{DK}
\textsc{M. M. Disconzi, M. A. Khuri,}
Compactness and non-compactness for the Yamabe problem on manifolds with boundary, J. Reine Angew. Math. \textbf{724} (2017), 145--201.

\bibitem{DMO}
\textsc{Z. Djadli, A. Malchiodi, M. Ould Ahmedou,}
Prescribing scalar and boundary mean curvature on the three dimensional half sphere. J. Geom. Anal. \textbf{13} (2003), 255-–289.

\bibitem{DMO2}
\bysame, The prescribed boundary mean curvature problem on $\B^4$. J. Differential Equations \textbf{206} (2004), 373--398.

\bibitem{Dr}
\textsc{O. Druet,}
Compactness for Yamabe metrics in low dimensions, Int. Math. Res. Not. \textbf{23} (2004), 1143--1191.

\bibitem{Es5}
\textsc{J. F. Escobar,}
Sharp constant in a Sobolev trace inequality, Indiana Univ. Math. J. \textbf{37} (1988), 687--698.

\bibitem{Es3}
\bysame, Uniqueness theorems on conformal deformation of metrics, Sobolev inequalities, and an eigenvalue estimate, Comm. Pure Appl. Math. \textbf{43} (1990), 857--883.

\bibitem{Es}
\bysame, Conformal deformation of a Riemannian metric to a scalar flat metric with constant mean curvature on the boundary, Ann. of Math. \textbf{136} (1992), 1--50.

\bibitem{Es4}
\bysame, The Yamabe problem on manifolds with boundary, J. Differential Geom. \textbf{35} (1992), 21--84.

\bibitem{Es2}
\bysame, Conformal metrics with prescribed mean curvature on the boundary, Calc. Var. Partial Differential Equations \textbf{4} (1996), 559--592.

\bibitem{FO}
\textsc{V. Felli, A. Ould Ahmedou,}
Compactness results in conformal deformations of Riemannian metrics on manifolds with boundaries, Math. Z. \textbf{244} (2003), 175--210.

\bibitem{FO2}
\bysame, A geometric equation with critical nonlinearity on the boundary, Pacific J. Math. \textbf{218} (2005), 75--99.

\bibitem{GM}
\textsc{M. Ghimenti, A. M. Micheletti,}
A compactness result for scalar-flat metrics on manifolds with umbilic boundary, preprint, arXiv:1903.10990.

\bibitem{GMP}
\textsc{M. Ghimenti, A. M. Micheletti, A. Pistoia}
Linear perturbation of the Yamabe problem on manifolds with boundary, J. Geom. Anal. \textbf{28} (2018), 1315--1340.

\bibitem{GMP2}
\bysame, Blow-up phenomena for linearly perturbed Yamabe problem on manifolds with umbilic boundary, J. Differential Equations \textbf{267} (2019), 587--618.

\bibitem{HL}
\textsc{Z.-C. Han, Y. Y. Li,}
The Yamabe problem on manifolds with boundary: existence and compactness results, Duke Math. J. \textbf{99} (1999), 485--542.

\bibitem{KMS}
\textsc{M. Khuri, F. Marques, R. Schoen,}
A compactness theorem for the Yamabe problem, J. Differential Geom. \textbf{81} (2009), 143--196.

\bibitem{KMW2}
\textsc{S. Kim, M. Musso, and J. Wei,}
Existence theorems of the fractional Yamabe problem, Anal. PDE \textbf{11} (2018), 75--113.

\bibitem{KMW}
\bysame, A compactness theorem of the fractional Yamabe problem, Part I: The non-umbilic conformal infinity, preprint, arXiv:1808.04951.

\bibitem{MN}
\textsc{M. Mayer, C. B. Ndiaye,}
Barycenter technique and the Riemann mapping problem of Cherrier-Escobar, J. Differential Geom. \textbf{107} (2017), 519--560.

\bibitem{Li}
\textsc{G. Li,}
A compactness theorem on Branson's $Q$-curvature equation, preprint, arXiv:1505.07692.

\bibitem{LX}
\textsc{Y. Y. Li, J. Xiong,}
Compactness of conformal metrics with constant $Q$-curvature. I, Adv. Math. \textbf{345} (2019), 116--160.

\bibitem{LZ}
\textsc{Y. Y. Li, L. Zhang,}
Compactness of solutions to the Yamabe problem II, Calc. Var. and Partial Differential Equations \textbf{25} (2005), 185--237.

\bibitem{LZ2}
\bysame, Compactness of solutions to the Yamabe problem III, J. Funct. Anal. \textbf{245} (2006), 438--474.

\bibitem{LZhu2}
\textsc{Y. Li, M. Zhu,}
Uniqueness theorems through the method of moving spheres, Duke Math. J. \textbf{80} (1995), 383--417.

\bibitem{LZhu}
\bysame, Yamabe type equations on three dimensional Riemannian manifolds, Comm. Contemp. Math. \textbf{1} (1999), 1--50.

\bibitem{Ma}
\textsc{F. C. Marques,} A priori estimates for the Yamabe problem in the non-locally conformally flat case, J. Differential Geom. \textbf{71}
(2005), 315--346.

\bibitem{Ma2}
\bysame, Existence results for the Yamabe problem on manifolds with boundary, Indiana Univ. Math. J. \textbf{54} (2005), 1599--1620.

\bibitem{Ma3}
\bysame, Conformal deformations to scalar-flat metris with constant mean curvature on the boundary, Comm. Anal. Geom. \textbf{15} (2007), 381--405.

\bibitem{Sc}
R. Schoen, Course notes on `Topics in differential geometry' at Stanford University, (1988), available at \verb"https://www.math.washington.edu/~pollack/research/Schoen-1988-notes.html"
\end{thebibliography}
\end{document}